\documentclass[10pt]{preprint}
\usepackage{color}

\usepackage{amsmath}
\usepackage{amssymb}
\usepackage{amsthm}
\usepackage{mhequ}
\usepackage{times}
\usepackage{microtype}

\usepackage{verbatim}

\numberwithin{equation}{section}

\newcommand{\eps}{\varepsilon}
\newcommand{\del}{\partial}

\newcommand{\vphi}{\varphi}

\newcommand{\norm}[1]{\| #1\|}

\newcommand{\Seps}{S_{\eps}}

\newcommand{\gen}{\mathcal{L}}
\let\CL\gen
\newcommand{\Ltilde}{\widetilde{\mathcal{L}}}

\def\eref#1{(\ref{#1})}
\newcommand{\innerprod}[1]{\langle#1\rangle}
\newcommand{\bigoh}{\mathcal{O}}
\newcommand{\rhobar}{\hat{\rho}}

\newcommand{\qbar}{\bar{q}}

\newcommand{\E}{\mathbb{E}}

\newcommand{\reals}{\mathbb{R}}
\let\R\reals
\newcommand{\integers}{\mathbb{Z}}

\newcommand{\naturals}{\mathbb{N}}

\newcommand{\complex}{\mathbb{C}}

\newcommand{\Linfty}{{L}^{\infty}}
\newcommand{\suptime}{\sup_{t \in [0,T]}}
\newcommand{\Ltwopi}{L^2[0,2\pi]}

\newcommand{\Aeps}{\mathcal{A}_\eps}

\newcommand{\tildeW}{\hat{W}}

\newcommand{\tildeu}{\hat{u}}
\newcommand{\tildex}{\widetilde{x}}
\newcommand{\tildey}{\widetilde{y}}

\newcommand{\hatW}{\hat{W}}
\newcommand{\hatu}{\hat{u}}

\theoremstyle{plain}
\newtheorem{lemma}[subsection]{Lemma}
\newtheorem{thm}[subsection]{Theorem}
\newtheorem{prop}[subsection]{Proposition}

\theoremstyle{definition}
\newtheorem{rmk}[subsection]{Remark}
\newtheorem{remark}[subsection]{Remark}
\newtheorem{ass}[subsection]{Assumption}
\newtheorem{corr}[subsection]{Corollary}

\def\scal#1{\langle #1\rangle}
\def\HS{{\mathrm{HS}}}
\def\BL{{\mathrm{BL}}}

\begin{document}
%
\title{Stochastic PDEs with multiscale structure}
\author{Martin Hairer$^1$ and David Kelly$^2$}
\institute{The University of Warwick, \email{M.Hairer@Warwick.ac.uk}
\and The University of Warwick, \email{D.T.B.Kelly@Warwick.ac.uk}}
\maketitle\thispagestyle{empty}


\begin{abstract}
We study the spatial homogenisation of parabolic linear stochastic PDEs exhibiting 
a two-scale structure both at the level of the linear operator and at the level of the Gaussian
driving noise. We show that 
in some cases, in particular when the 
forcing is given by space-time white noise, it may happen that the homogenised SPDE is 
not what one would expect from existing results for PDEs with more regular forcing terms. 
\end{abstract}

\section{Introduction}
In the material sciences, there is a significant interest towards objects that contain one structure at a macroscopic scale, overlaying a totally different structure on a microscopic scale. Examples range from everyday life, such as concrete and fibreglass, to the cutting edge of science, such as the cloaking devices implemented by meta-materials. Composite materials pose an important mathematical problem. Given a system with certain dynamics on a macroscopic scale and separate, but not necessarily independent, dynamics on a microscopic scale, approximate the effective dynamics of the whole system when the microscopic scale is small. Such problems can be formulated, and dealt with, using homogenisation theory, see for example \cite{MR0163062,papa77,vandensome,temam00,MR2148377}, as well as the monographs \cite{bensoussan78,pavliotis08} 
and references therein.  

The following is the prototypical homogenisation problem. Take a Markov process $X$ on $\R$ with generator 
\begin{equation}\label{e:unscaledgenerator}
\gen= b(x)\del_x + \frac{1}{2}\sigma^2(x) \del_x^2\;,
\end{equation}  
where $b$ and $\sigma$ are suitably smooth functions, periodic on $[ 0,2\pi]$. Consider then the diffusively rescaled process $X_\eps(t) = \eps X(t/\eps^2)$, with generator given by  
\begin{equation}\label{e:generator}
\gen_\eps= \frac{1}{\eps}b(x/\eps)\del_x + \frac{1}{2}\sigma^2(x/\eps) \del_x^2\;.
\end{equation} 
We also require that $\sigma$ is bounded away from zero and that the ``centering condition''
$
\int_0^{2\pi} b(v)/\sigma^2(v) dv=0
$
is satisfied. 

One example to keep in mind is the when $\sigma =1$ and 
\[
V(x/\eps) = - \int^{x/\eps} b(v) dv \;.   
\]
The centering condition guarantees that $\int_0^{2\pi} b(v) dv =0$, so that $V(x/\eps)$ itself is $2\pi \eps$ periodic. 
In this case, the diffusion $X_\eps$ provides a simple model for diffusion in a one-dimensional composite material, where the material is composed of cells of size $2\pi \eps$ and the dynamics in each cell is governed by the potential $V(x/\eps)$. 

It is a classical result that
\begin{equation}\label{e.BMresult}
X_\eps(t) \Rightarrow \mu B(t)\;,
\end{equation}
where $B(t)$ is a Brownian motion on $\R$, $\mu>0$ is a constant determined by $b$ and $\sigma$, and $\Rightarrow$ denotes convergence in distribution on the space of continuous functions \cite{bensoussan78}. This result is powerful when analysing parabolic PDEs of the following type 
\begin{equation}\label{e.parabolic}
\del_t u_\eps(x,t) = \gen_\eps u_\eps(x,t) + f(x,t)\;,
\end{equation}
with some forcing term $f$. 
We will assume $u_\eps(x,0)=0$ as we are more interested in the forcing term. Duhamel's principle then states that 
\[
u_\eps(x,t) = \int_0^t \E[f(s,X_\eps(t-s))| X_\eps(0)=x]\,ds\;,
\]
where  $\E$ averages over the paths $X_\eps$ (but not any possible randomness in the forcing term). If $f$
is sufficiently regular, it follows from \eqref{e.BMresult} that  $u_\eps \to u$ as $\eps \to 0$, where 
$u$ satisfies the PDE
\begin{equation}\label{e.naive}
\del_t u(x,t) = {\mu\over 2} \del_x^2 u(x,t) + f(x,t)\;.
\end{equation}
Such results have been widely generalised in both the forcing terms considered and also the structural assumptions placed on the generator $\gen_\eps$, see for example \cite{pardoux99,MR2078542,pardoux08,pardoux09}. The article \cite{pardoux09} contains a brief but recent overview of the field.
On the other hand, one can find only very few results in the literature treating the case of stochastic PDEs
where both the noise term and the linear operator exhibit  a multiscale structure, and this is the main focus of
this article. In some situations where the limiting noisy term is sufficiently regular, the previously mentioned
results have been extended to the stochastic case, see for example \cite{MR2072382,duan07,wang07}.
The present article aims to provide a preliminary understanding of the type of phenomena that can arise
in the situation where the limiting equation is driven by very rough noise, so that resonance effects can also
play an important role.

Over the last few decades, there has been much progress towards making sense of solutions to stochastic PDEs, where the forcing term may be a highly irregular Gaussian signal taking values in spaces of rather irregular distributions, see
for example \cite{daprato92,hairer09} for introductory texts on the subject. It is therefore natural to ask whether asymptotic results for PDEs like \eqref{e.parabolic} can be extended to the case where $f$ is a random, distribution-valued 
process. To give an idea of the type of results obtained in this article, let $\xi$ be space-time white noise,
which is the distribution-valued Gaussian process formally satisfying $\E \xi(s,x) \xi(t,y) = \delta(s-t)\delta(x-y)$.
For fixed $\eps>0$, one can easily show that 
\begin{equation}\label{e.STWN}
\del_t u_\eps = \gen_\eps u_\eps + \xi
\end{equation}
has a unique solution $u_\eps$ with almost surely continuous sample paths in $\Ltwopi$. By analogy with the
 classical theory outlined above and since $\xi$ does not show any explicit $\eps$-dependence, 
 one might guess that $u_\eps$ has a limit $u$, satisfying 
\begin{equation}\label{e:classical}
\del_t u = \mu \del_x^2 u + \xi\;.
\end{equation} 
It turns out that this is not the case. Instead, we will show that the true limit solves 
\begin{equation}\label{e.surprise}
\del_t u = \mu \del_x^2 u + \norm{\rho} \xi\;, 
\end{equation}
where $\norm{\cdot}$ denotes the $\Ltwopi$ norm (normalised such that the corresponding scalar product is given by
$\scal{f,g} = {1\over 2\pi}\int_0^{2\pi}f(x)g(x)\,dx$) and $\rho$ is the invariant measure for the process with generator $\gen$, normalised to satisfy $\innerprod{\rho,1}=1$. 

\begin{rmk}
By Jensen's inequality, one always has $\norm{\rho} \ge 1$, with equality if and only if $\rho$ is constant. As
a consequence, \eref{e.surprise} differs from \eref{e:classical} as soon as $\CL$ is not in divergence form.
Furthermore, the effect of the noise is always enhanced by non-trivial choices of $\CL$, which is 
a well-known fact in different contexts \cite{pavliotis08}.
\end{rmk}

The crucial fact is of course the lack of regularity of $\xi$. Since the law of the process $X_\eps$ generated by $\gen_\eps$ will vary with $x/\eps$, its law will typically have large Fourier components at wave numbers 
close to integer multiples of $1/\eps$. The difference between \eref{e.surprise} and \eref{e:classical} can then
be understood, at least at an intuitive level, as coming from the resonances between these Fourier modes and
the corresponding Fourier modes of the driving noise.
Such resonances would be negligible for more regular noises, but turn out to lead to non-negligible contributions
in the case of space-time white noise. 

The aim of this article is to investigate this phenomenon for SPDEs of the type \eqref{e.STWN}, but replacing $\xi$ with a more general Gaussian forcing term. In particular, we treat noise that exhibits spatial structure 
at the microscopic scale. We can always (formally) write such signals as 
\begin{equation}\label{e:reprNoise}
\zeta(x,x/\eps,t) =  \sum_{k\in\integers} q^k(x,x/\eps)\dot{W_k}(t)\;,
\end{equation}
where the $W_k$ are i.i.d.\ complex-valued Brownian motions, save for the condition $W_{-k} = W_k^\star$
ensuring that the overall signal is 
real-valued. Throughout this article, we will require the additional assumption that the noise $\zeta$ is \emph{cell-translation invariant}, in the sense that its distribution is unchanged by translations by multiples of $2\pi \eps$. 
This assumption reflects the idea that the underlying material has the same structure in each cell. 
At the level of the representation \eref{e:reprNoise}, this invariance is enforced by assuming that 
one has 
\begin{equation}\label{e:reprq}
q^k(x,x/\eps) = q_k(x/\eps)e^{ikx}\;,
\end{equation}
for each $k \in \integers$, where $\{q_k\}$ is a collection 
of $2\pi$-periodic functions. 

To see that this leads to the claimed invariance property, notice that, for $x,y$ satisfying $x-y=2\pi \eps n$, we have that
\begin{align*}
\sum_{k\in\integers} q_k(y/\eps) e^{iky}\dot{W}_k(t) &= \sum_{k\in\integers} q_k(x/\eps) e^{ikx}e^{2\pi i k \eps n}\dot{W}_k(t)\\
&\stackrel{d}{=} \sum_{k\in\integers} q_k(x/\eps) e^{ikx}\dot{W}_k(t)\;.
\end{align*}
Indeed, since $W_k$ is a complex Brownian motion, rotating it by $2\pi k\eps n$ does not change its distribution.
Conversely, cell-translation invariance of the noise is equivalent to the fact that its covariance operator $C_\eps$
commutes with the translation operator $T_\eps$ given by $T_\eps f(x) = f(x+2\pi \eps)$.
The spectrum of $T_\eps$ consists of $\{e^{ik\eps}\,:\, k \in \integers\}$, with corresponding eigenspaces given
by $V_k = \{q(x/\eps) e^{ikx}\}$, where $q$ is periodic with period $2\pi$. As a consequence, there is
no loss of generality in assuming the representation \eref{e:reprq}.

Thus, we restrict our attention to the following class of SPDEs, written in the notation of \cite{daprato92}:
\begin{equation}
du_\eps(x,t) = \gen_\eps u_\eps(x,t) dt + \sum_{k \in \integers} q_k(x/\eps)e^{ikx}dW_k(t)\;.\label{e:SPDE1} 
\end{equation}
Again, we will always assume that $u_\eps$ satisfies periodic boundary conditions on $[0,2\pi]$. By linearity,
we can and will restrict ourselves to the case of vanishing initial conditions.
We will always assume certain regularity conditions on $b$ and $\sigma$, as well as a centering condition, which is a standard requirement of homogenisation problems. This is detailed in Assumption \ref{ass.elliptic} below.
\begin{rmk}\label{rmk.geometry}
Unlike several recent studies \cite{duan07,wang07} we do not consider periodically perforated spatial domains. Instead, we assume that our domain $[0,2\pi]$ has been split into cells of size $2\pi \eps$ and that diffusions behave identically in each cell. This is implemented through the periodicity of $b$, $\sigma$ and $q_k$. Thus, all composite-type geometry comes through the periodicity of the generator $\gen_\eps$ and the infinite dimensional noise; the spatial domain $[0,2\pi]$ does not depend on $\eps$ in any way. However, we do require that the domain be partitioned in to cells of size $2\pi \eps$. It is therefore natural to require that $\eps^{-1} \in \naturals$ so that $[0,2\pi]$ contains an integer number of cells.  
\end{rmk}
We have already seen that taking $q_k=1$ results in the surprising limit \eqref{e.surprise}. However, if we chose $q_k=|k|^{-1}$ then the forcing term would be a continuous Gaussian process in $\Ltwopi$, and by classical results $u_\eps$ would converge to the unsurprising limit, as in \eqref{e.naive}. We would like to classify those choices of $q_k$ that result in the surprising limit, and those that result in the unsurprising limit.  
\par
Firstly, we will identify a large class of signals that result in the unsurprising limit. In particular, these signals need not be continuous processes in $\Ltwopi$. To guarantee the unsurprising limit, we need some control over the coefficients of the noise $q_k$ when $k$ is large, as well as a suitable regularity assumption. If we assume that the coefficients decay algebraically as $k\to \infty$, then we are able to show that solutions converge to the correct limit and that this convergence occurs in $L^2(P)$. In particular, the quantity $\norm{q_k}$ must decay like $|k|^{-\alpha}$ as $k\to\infty$, for some $\alpha \in (0,1)$. The precise condition is detailed in Assumption \ref{ass.strongnoise}. With these conditions in place, we will prove the following.
\begin{thm}\label{thm:formulation1}
Suppose the SPDE \eqref{e:SPDE1} satisfies Assumptions \ref{ass:elliptic} and \ref{ass.strongnoise}. Then the solutions $u_\eps$ converge to the solutions of
\begin{equation}\label{e:limiting1}
du(x,t) = \mu \del_x^2 u(x,t) dt + \sum_{k \in \integers} \innerprod{q_k,\rho} e^{ikx} dW_k(t) \;,  
\end{equation}
in the sense that there exists $C_T>0$ and $\theta > 0$ such that
\[
\E \suptime |\innerprod{u_\eps(t)-u(t),\vphi}|^2 \leq C_T\eps^{\theta} \;,
\]
for all $\vphi \in H^{s}$ with large enough $s$.  
\end{thm}
\begin{rmk}\label{rmk.HS}
Past results \cite{MR2072382,duan07} rely on the noise being Hilbert-Schmidt in the sense that
\[
\sum_k \norm{q_k}^2 < \infty\;.
\]
It is important to note that this condition does not imply our condition on the $\norm{q_k}$. Indeed, one can easily exhibit a sufficiently sparse sequence $\norm{q_k}$ that is square summable but which only converges logarithmically to zero. On the other hand, there are many situations where the noise is not Hilbert-Schmidt, that do fall into our framework. With only the Hilbert-Schmidt assumption, one can still prove via a tightness argument that the SPDE \eqref{e:SPDE1} has a weak limit and apply homogenisation techniques, similar to those found in \cite{duan07}, to show that the limiting SPDE is indeed \eqref{e:limiting1}. However, we will not treat this case as it is somewhat incongruous with the existing framework.
\end{rmk}
\begin{rmk}
Although not immediately clear, this is indeed the unsurprising limit in the sense of \eqref{e.naive}. To see this, pick $q_k(x/\eps)=\hat{q}_k |k|^{-\alpha}$. It is easy to see that, since $\innerprod{\rho,1}=1$, the noise in the limiting SPDE \eqref{e:limiting1} is the same as the original noise, as was the case in the classical result \eqref{e.naive}.   
\end{rmk}
This result is reminiscent of previous results \cite{duan07,wang07}, but stronger in the sense that genuine mean-squared convergence is obtained. Moreover, the result comes with rates of convergence. These are some of the perks enjoyed by a Fourier analytic framework, which we employ in place of the tightness arguments usually found in homogenisation problems. Of course, we still have weak convergence in a variational sense.
\par
There are some important things to note concerning the limiting SPDE \eqref{e:limiting1}. Firstly, it is a stochastic heat equation with additive noise, and that noise comes with the same spatial regularity as the noise in the original SPDE. That is, the coefficients of $W_k$ decay with the same rate. Secondly, if we choose the noise to satisfy the centering condition $\innerprod{q_k,\rho}=0$ for each $k\in\integers$, then the solution $u_\eps$ will converge strongly to zero as $\eps \to 0$. In other words, the presence of noise will have vanishingly small effect on the system \eqref{e:SPDE1} when $\eps$ is small. It is natural to ask whether we can find the largest vanishing term as $\eps \to 0$. To obtain this term, we scale up the solution $u_\eps$ by some cleverly chosen inverse factor of $\eps$ and then seek a non-zero solution. For this procedure to work, we need to have very precise control over the coefficients $q_k$ when $k$ is large. Namely, we require that there exists some $\alpha \in (0,1)$ and a sufficiently regular function $\qbar$ such that $|k|^\alpha q_k \to \qbar$ in $\Ltwopi$ as $|k|\to \infty$. One can check that these assumptions imply those made for the previous theorem. The precise assumptions are detailed in Assumption \ref{ass.weaknoise}. With these conditions, we can prove the following.
\begin{thm}\label{thm:formulation2}
Suppose the SPDE \eqref{e:SPDE1} satisfies Assumptions \ref{ass:elliptic} and \ref{ass.weaknoise} for some decay exponent $\alpha \in (0,1)$ and $\innerprod{q_k,\rho}=0$ for all $k\in \integers$. Then there exists a process $\tildeu_\eps$ equal in law to $u_\eps$ but defined on a different probability space, such that the rescaled solutions $\eps^{-\alpha} \tildeu_\eps$ converge to the solutions of 
\begin{equation}\label{e.limiting2}
dv(x,t) = \mu \del_x^2 v(x,t) dt + \norm{\qbar \rho}_{-\alpha} \sum_k e^{ikx} d\tildeW_k(t)   
\end{equation}
in the sense that 
\[
\lim_{\eps \to 0} \E \suptime |\innerprod{\eps^{-\alpha}\tildeu_\eps(t)-v(t),\vphi}|^2=0\;,
\]
for all $\vphi \in H^{s}$ with large enough $s$. 
\end{thm}
Here the convergence result is weak in both a variational and probabilistic sense. In general, nothing stronger is possible. Although the result looks like convergence in mean-squared, it is merely \emph{disguised} convergence in law
since we must define the limiting solution on a different probability space to the original SPDE. Such results are often obtained artificially using the Skorokhod embedding theorem. In our case however, this is the natural way to write down the result. In particular, for fixed $\eps>0$, the dependencies of $\tildeW_m$ can be traced back to the original BMs. It is worth mentioning that the scaling factor required in order to find this term is in fact $\eps^{-\alpha}$, which is precisely the amount of decay placed on the coefficients $q_k$. In the limiting SPDE \eqref{e.limiting2}, we use the notation 
\[
\norm{f}_{-\alpha} = \left(\sum_{k\in\integers}|k|^{-2\alpha} |\innerprod{f,e_k}|^2 \right)^{1/2}\;,
\]
where $e_k(x) = e^{ikx}$. 

As before, there are several things to note about the SPDE \eqref{e.limiting2}. Firstly, it is again a stochastic heat equation with additive noise, but now all contributions from the original driving noise come from the very high modes, as indicated by the factor $\norm{\qbar \rho}_{-\alpha}$. Thus, the coefficients $q_k$ with low $k$ have no bearing at all on the limit. In particular, if one wanted to approximate the noise by cutting off the sum at a large value of $k$, they would be making a drastic mistake! Moreover, this suggests that $v$ arises due to constructive interference occurring in the very high modes of the noise. The second observation to make is that no matter what spatial regularity is possessed by the noise in the original SPDE, the limiting SPDE is always driven by space-time white noise. As one might guess, the factor $\eps^{-\alpha}$ essentially scales away the decay on the coefficients $q_k$ and hence destroys the regularity of the driving noise.  

The previous theorem may seem a bit off topic, as we are trying to determine how choices of $q_k$ affect the limiting SPDE. However, the following theorem tells us that the second order term found in Theorem \ref{thm:formulation2} acts as the \emph{bridge} between the surprising limit and the unsurprising limit. 
In particular, we will show that the surprising limit occurs precisely when this second order term becomes non-vanishing. We can see in \eqref{e.STWN} that space-time white noise falls into the `$\alpha=0$ class', in the context of the previous theorems, since obviously $q_k=1$ does not decay. Since the second order term was shown to be $\bigoh(\eps^{\alpha})$, one would expect this term to become $\bigoh(1)$ and hence contribute to the limit in the space-time white noise case. This suggests that the second order term is precisely the difference between the surprising limit and the unsurprising limit. The following theorem proves this to be the case not just for \eqref{e.STWN} but for all SPDEs driven by noise in the $\alpha=0$ class. 

The only added requirement for noise to be in this class is that there exists $\qbar \in H^1$ such that $q_k\to\qbar$ as $k\to \infty$ and that this convergence happens with fast enough rate. The precise conditions are found in Assumption \ref{ass.STWN}. We have that following result.     
\begin{thm}\label{thm.formulation3}
Suppose the SPDE \eqref{e:SPDE1} satisfies Assumptions \ref{ass:elliptic} and \ref{ass.STWN}. Then there exists $\hatu_\eps$ equal in law to $u_\eps$, but defined on a different probability space, such that $\hatu_\eps$ converges to the solutions of
\begin{equation}\label{e:limiting3}
d\hatu(x,t) = \mu \del_x^2 \hatu(x,t) dt + \sum_{k \in \integers} (|\innerprod{q_k,\rho}|^2 - |\innerprod{\qbar,\rho}|^2 + \norm{\qbar \rho}^2)^{1/2} e^{ikx} d\hatW_k(t)   \;,
\end{equation}
in the sense that 
\[
\lim_{\eps \to 0} \E \suptime |\innerprod{\hatu_\eps(t)-\hatu(t),\vphi}|^2 =0 \quad \text{for all $\vphi \in H^{s}$}
\]
for large enough $s$.  
\end{thm}
As one might expect, this result is almost a combination of the two previous results, only a few extra ingredients are needed to prove it. In the $\norm{\cdot}_{-\alpha}$ notation of Theorem \ref{thm:formulation2}, we have that
\[ -|\innerprod{\qbar,\rho}|^2 + \norm{\qbar \rho}^2 = \norm{\qbar \rho}_0^2\;, \] 
which is precisely the contribution from the second order term (squared), so that \eqref{e:limiting3} really is a combination of the first order limit in \eqref{e:limiting1} and the second order limit in \eqref{e:limiting3}. Note that instead of the noise being comprised of the sum of the first order and second order terms, we have the square-root of the sum of the squares. This is simply because we want to write each term in the noise as a single Gaussian, rather than a sum of two independent Gaussians. Just as in Theorem \ref{thm:formulation2}, the BMs $\hatW_m$ are, for fixed $\eps>0$ defined in terms of the original BMs.
\par
To prove these three convergence results, we develop several tools that are useful when dealing with any SPDE whose underlying diffusion is driven by $\gen_\eps$. Firstly, we develop a relationship between the interpolation spaces generated by $\gen_\eps$ and the usual Sobolev spaces. This is useful in determining which function spaces contain our solutions (uniformly in $\eps$) and furthermore determining where convergence occurs. Secondly, we show that the effect of the semigroup $\Seps$ generated by $\gen_\eps$ on a certain class of functions is approximated well by the heat semigroup. This is akin to the well-known fact that $\gen_\eps \Rightarrow \mu\del_x^2$, as discussed earlier.\par
The article is structured in the following way. In Section \ref{s.formulation}, we give a precise formulation of the main SPDE and detail the structural assumptions. In Section \ref{s.preliminary} we develop some tools necessary for the proof of the convergence theorems. In Section \ref{s:convergence} we rigorously state and prove all three convergence theorems.

\subsection*{Acknowledgements}

{\small
We would like to thank T.~Souganidis for pointing out the type of convergence considered in
Corollary~\ref{corr:convolution}.
Financial support for MH was kindly provided by EPSRC grant EP/D071593/1, by the 
Royal Society through a Wolfson Research Merit Award, and by the Leverhulme Trust through a Philip Leverhulme Prize.
DK was supported by a Warwick Postgraduate Research Scholarship.
}

\section{Formulation of the SPDE and some notation}\label{s.formulation}

Recall that $\Ltwopi$ denotes the complex $L^2$ space
with its inner product normalised as 
\[
\innerprod{f,g} = \frac{1}{2\pi} \int_0^{2\pi} f g^* dx  \;,
\]
and corresponding norm $\norm{\cdot}$. We denote elements of the orthonormal Fourier basis by $e_k(x) = e^{ikx}$. We will also denote the usual $L^{\infty}$ norm by $\norm{\cdot}_{\infty}$. We define $\mathcal{C}_b^2$ as the subspace of $\Ltwopi$ of bounded, continuous functions with two bounded, continuous derivatives. We measure regularity through the Sobolev spaces $H^s$ which we define as the completion of $\Ltwopi$ under the norm
\[
\norm{\cdot}_{H^s} = \norm{(1-\del_x^2)^{s/2} \cdot}\;,
\]
for any $s \in \reals$. We shall also make use of the following Sobolev-like semi-norm
\begin{equation}\label{e:seminorm}
\norm{f}_{-s} = \left( \sum_{k \in \integers} |k|^{-2s} |\innerprod{f,e_k}|^2   \right)^{1/2}\;,
\end{equation}
which can only be defined on $f$ with $\innerprod{f,1}=0$. One can therefore think of this semi-norm as the norm $\norm{(-\del_x^2)^{-s}\cdot}$ defined on the space of mean-zero functions. We denote by $\norm{\cdot}_\HS$ the  Hilbert-Schmidt norm on linear operators that map $\Ltwopi$ into itself. As a shorthand we will write 
\[
f^\eps (x) = f(x/\eps)\;,
\]
when we want to omit the function's dependence on $x$. Finally, we will use the notation $f \lesssim g$ to imply that $|f/g|$ can be bounded by some constant that is independent of parameters involved in the expression. The precise independence will be clear from the context. 
\subsection{Formulation of the equation}\label{ss:formulation}
Let $b$ and $\sigma$ be twice continuously differentiable $2\pi$-periodic functions and define the differential operator $\gen_\eps$ as in \eqref{e:generator} and likewise define the unscaled operator $\gen$ as in \eqref{e:unscaledgenerator}. Following \cite{papa77,bensoussan78}, we require some conditions on the generator $\gen_\eps$ for the homogenization problem to have a limit.

\begin{ass}\label{ass:elliptic}\label{ass.elliptic}
Assume that $b, \sigma \in \mathcal{C}_b^2$ and that the centering condition
\begin{equation}
\int_0^{2\pi} \frac{b(x)}{\sigma^2(x)} dx =0\;,
\end{equation}
is satisfied. Furthermore, $\sigma$ is uniformly elliptic, namely
\begin{equation}\label{e:densitybound}
0< \delta < \sigma(x) < \delta' < \infty \;,
\end{equation}
for some fixed $\delta$ and $\delta'$. 
\end{ass}

\begin{remark}
One can check that the centering condition implies that
\begin{equation}
\int_0^{2\pi} b(x)\rho(x)dx = 0\;,
\end{equation}
where $\rho$ is the solution to $\gen^* \rho =0$ with periodic boundary conditions and satisfying $\innerprod{\rho,1}=1$. We will call $\rho$ the invariant density for $\gen$, despite the fact that it is not normalised to be a probability measure. This centering condition serves the same purpose as subtracting the mean when trying to obtain a central limit theorem.
\end{remark}

\begin{remark}
The smoothness of $b$ and $\sigma$, combined with the ellipticity condition, are sufficient to guarantee that $\rho \in \mathcal{C}_b^2$ and similarly for all positive and negative powers of $\rho$.
\end{remark}

Our main object of interest is the following SPDE, defined on finite temporal and spatial domains
\begin{align}
du_\eps (x,t)& =  \gen_\eps u_\eps(x,t) dt + \sum_{k \in \integers} q_k(x/\eps) e_k(x)dW_k(t)
\quad &&\text{$(x,t)\in [0,2\pi]\times(0,T]$} \label{e:formulationSPDE} \\
u_\eps(0,t)& =  u_\eps(2\pi,t) \quad && {t \in [0,T]}\\
u_\eps(x,0)& =  0 \quad && {x \in [0,2\pi]}\;.
\end{align}
Each $q_k(\cdot)$ is a continuous $2\pi$-periodic element of $\Ltwopi$, taking values in $\reals$ and we require that $q_{-k} =q_k$ for each $k\in\integers$. As stated in Remark \ref{rmk.geometry}, the microscopic parameter $\eps \in (0,1)$ must satisfy $\eps^{-1} \in \naturals$. We define the sequence of Brownian motions $\{W_k \}_{k\in\integers}$ in the following way: $W_0$ is a $\reals$-valued BM, where as $\{W_k\}_{k \geq 1}$ are $\complex$-valued BMs, and $\{W_k\}_{k\geq0}$ are pairwise independent; we then set $W_{-k} = W_k^{\star}$, where $(\cdot)^{\star}$ denotes complex conjugation. Every bi-infinite sequence of Brownian motions considered in the sequel will satisfy this conjugation property. As stated, we assume periodic boundary conditions and take the initial condition to be identically zero. We choose this initial condition as we are only interested in the evolution of the noise through the system. Determining the evolution with a non-trivial initial condition is equivalent to adding the solution to the noiseless problem, which has been well studied \cite{bensoussan78, papa77, pavliotis08}. 

For convenience we introduce the linear operator on $\Ltwopi$ by
\begin{equation}\label{e:Qeps}
Q_\eps e_k (x) = q_k(x/\eps)e_k(x)\;,
\end{equation}
and one can then represent the noise in \eqref{e:formulationSPDE} as $Q_\eps dW$ where $dW$ denotes space-time white noise. We shall now list the assumptions needed to prove Theorems \ref{thm:formulation1}, \ref{thm:formulation2}, \ref{thm.formulation3} respectively. Firstly, we require the following condition to prove Theorem \ref{thm:formulation1}. 
\begin{ass}
\label{ass.strongnoise}
There exists $\alpha \in (0,1)$ such that 
\begin{equation}
\norm{q_k} \lesssim 1\wedge|k|^{-\alpha} \;,
\end{equation}
for each $k\in\integers$. Moreover, if $\alpha \in (0,1/2]$ then we additionally require that 
\begin{equation}
\sup_{k\in\integers}\norm{\qbar_k}_{H^1} < \infty \;,
\end{equation}
where $\qbar_k = q_k / \norm{q_k}$. 
\end{ass}
To prove Theorem \ref{thm:formulation2}, we need slightly different assumptions to those required for Theorem \ref{thm:formulation1}. Namely, we need the following.
 \begin{ass}\label{ass.weaknoise}
There exists $\alpha \in (0,1)$ and $\qbar \in \Ltwopi$ such that 
\begin{equation}\label{e.assweak1}
\lim_{k \to \pm \infty} \norm{ |k|^{\alpha} q_k - \qbar } = 0\;.
\end{equation}
Moreover, if $\alpha \in (0,1/2]$ then we additionally require that
\begin{equation}
\sup_{k\in\integers} \norm{\qbar_k}_{H^1} < \infty \;.
\end{equation}
\end{ass}
Note that \eqref{e.assweak1} guarantees that the bound 
\[
\norm{q_k} \lesssim 1\wedge |k|^{-\alpha}
\]
holds for all $k\in\integers$ and therefore Assumption \ref{ass.weaknoise} implies Assumption \ref{ass.strongnoise}. Unlike in Theorem \ref{thm:formulation1}, having a rate of decay on $q_k$ does not suffice, we now need precise control over how $q_k$ tends to zero as $k\to \infty$.
\par
Recall that Theorem \ref{thm.formulation3} deals with those SPDEs that converge to the so called wrong limit. We claimed that this wrong limit occurred when the limit from Theorem \ref{thm:formulation1} combined with the limit from Theorem \ref{thm:formulation2}, by formally taking $\alpha=0$. Since Assumption \ref{ass.weaknoise} implies Assumption \ref{ass.strongnoise}, our condition on the noise for Theorem \ref{thm.formulation3} should look like Assumption \ref{ass.weaknoise}, with $\alpha=0$. Actually, we need a tiny bit more than this.
\begin{ass}\label{ass.STWN}
We require that there exists $\qbar \in H^1$ and $\eta \in [0,1)$ such that 
\begin{equation}
\sum_{k\in\integers} (1\wedge|k|^{-\eta}) \norm{q_k-\qbar}_{H^1}^2 < \infty\;.
\end{equation}
\end{ass}
At first glance this looks quite a bit stronger than Assumption \ref{ass.weaknoise} with $\alpha=0$. However, Assumption \ref{ass.weaknoise} with $\alpha=0$ implies that $\norm{q_k -\qbar}_{H^s} \to 0$ for every $s<1$, since the convergence is true in $\Ltwopi$ and the sequence $\{q_k \}$ is uniformly bounded in $H^1$. And since $\eta$ can be arbitrarily close to $1$, Assumption \ref{ass.weaknoise} almost implies Assumption \ref{ass.STWN}, but not quite. Note that the uniform boundedness condition on $\norm{q_k}_{H^1}$ is not implicitly stated, but it is implied by the listed assumptions. The parameter $\eta$ will affect the strength of the convergence result in Theorem \ref{thm.formulation3}, namely, larger $\eta$ leads to weaker convergence.  
\begin{rmk}
Another sufficient condition for Theorem \ref{thm.formulation3} is that 
\begin{equation}\label{e.alternative}
\sum_k \norm{q_k -\qbar}^2 < \infty\;,
\end{equation}
with $\qbar \in H^1$. Actually, we could also replace the regularity condition in Assumption \ref{ass.strongnoise} with \eqref{e.alternative}. However we consider the regularity assumption to be a more natural choice.      

\end{rmk}
\par
We define solutions to \eqref{e:formulationSPDE} using the mild formulation
\begin{equation}\label{e.mild}
u_\eps(x,t) = \int_0^t \Seps(t-s)Q_\eps dW(s) = \sum_{k\in\integers} \int_0^t S_\eps(t-s)q_k(x/\eps)e_k(x)dW_k(s)\;,
\end{equation}
where $S_\eps(t)$ is the semigroup generated by $\gen_\eps$. It is easy to check, using techniques introduced in the next section, that for fixed $\eps>0$, the semigroup $S_\eps(t)$ is a $\mathcal{C}_0$-semigroup. In this case, one can check that weak and mild solutions coincide \cite{hairer09,daprato92}, so the mild solution is indeed the correct one to look at. We also have the following regularity result
\begin{prop}\label{prop.solution}
Suppose Assumptions \ref{ass:elliptic}, \ref{ass.strongnoise} or \ref{ass:elliptic}, \ref{ass.STWN} hold true. Then, for fixed $\eps \in (0,1)$, the solution $u_\eps$ to \eqref{e:formulationSPDE} has almost surely continuous sample paths in $\Ltwopi$. 
\end{prop}
\begin{proof}
Using standard results for linear SPDEs \cite{hairer09,daprato92} we need only check that 
\[
\norm{S_\eps(t)Q_\eps}_\HS < \infty  \;,
\]
for every $t \in (0,T]$ and that there exists $\beta \in (0,1/2)$ such that
\[
\int_0^T t^{-2\beta} \norm{S_\eps(t)Q_\eps}_\HS^2 dt < \infty\;.
\]
In Lemma \ref{lem:apriori} below, we show that Assumption \ref{ass.strongnoise} implies that 
\[
\norm{S_\eps(t)Q_\eps}_\HS \lesssim \eps^{-4\gamma}|t|^{-\gamma} \left( \sum_{k\in\integers} (1\wedge |k|^{-4\gamma}) \norm{q_k}_{H^1}^2 \right)^{1/2} \;,  
\]
for any $\gamma \in (0,1/2)$. In Lemma \ref{lem.apriori2}, we show that Assumption \ref{ass.STWN} implies a similar estimate. The result follows immediately.
\end{proof}
\begin{rmk}
Note that although the decay assumption on $\norm{q_k}$ was not needed to show regularity of the solutions, it is necessary when proving convergence as $\eps \to 0$. It furthermore allows us to fine tune our results so that we can find the optimal space in which convergence occurs. 
\end{rmk}
\section{Preliminary Results}\label{s.preliminary}
In this section we shall develop a few tools necessary for the proof of the main results. In Section~\ref{ss.semigroup}, we start with some standard results concerning the semigroups generated by one dimensional It\^o diffusions. In Section~\ref{ss.interpolation}, we develop a relationship between the interpolation spaces of $\gen_\eps$ and the Sobolev spaces. Finally, in Section~\ref{ss.estimating}, we go on to approximate the effect of the adjoint semigroup $\Seps^*(t)$ on trigonometric polynomials. 
\subsection{Properties of the diffusion}\label{ss.semigroup}
We recall some basic results concerning the semigroup $\Seps(t)$ generated by $\gen_\eps$. Firstly, we have the following smoothing properties.
\begin{lemma}\label{l.semigroup}
For any $t \in [0,T]$ we have that
\begin{equation}\label{e:boundedsemigroup}
\norm{\Seps(t)} \leq C_T\;.
\end{equation}
Moreover, for any $\gamma \in [0,1)$ we have that
\begin{equation}\label{e:smoothingsemigroup}
\norm{(1-\gen_\eps)^{\gamma}\Seps(t)} \lesssim t^{-\gamma} \;.
\end{equation}
Finally, the same results hold true with $S_\eps(t)$ and $\gen_\eps$ replaced with their adjoints $S^*_\eps(t)$ and $\gen_\eps^*$.
\end{lemma}
\begin{proof}
We shall only prove \eqref{e:smoothingsemigroup} since \eqref{e:boundedsemigroup} follows as a special case. If $\gen_\eps$ were self-adjoint, then the result would follow easily from the spectral theorem \cite{hairer09}. $\gen_\eps$ is self-adjoint if the domain of the operator is taken to be the weighted space $L^2(\rho_\eps)$ with norm $\norm{f}_{\rho_\eps} = \norm{ f \rho_\eps^{1/2} }$ and corresponding inner product, where $\rho_\eps$ is the invariant density for $\gen_\eps$. The spectral theorem therefore implies that
\[
\norm{(1-\gen_\eps)^{\gamma}\Seps(t)f}_{\rho_\eps}\lesssim  t^{-\gamma}\norm{f}_{\rho_\eps}\;.
\]
Furthermore, one can easily show that $\rho_\eps = \rho(x/\eps)$ where $\rho$ is the invariant density of $\gen$, which we assumed in \eqref{e:densitybound} to be bounded above and away from zero. We therefore have that 
\begin{align*}
\norm{(1-\gen_\eps)^{\gamma}\Seps(t)f} &\leq\norm{\rho^{-1/2}}_{\infty} \norm{(1-\gen_\eps)^{\gamma}\Seps(t)f}_{\rho_\eps}\\ 
&\lesssim 
 t^{-\gamma} \norm{\rho^{-1/2}}_{\infty} \norm{f}_{\rho_\eps} \leq  t^{-\gamma} \norm{\rho^{-1/2}}_{\infty} \norm{\rho^{1/2}}_{\infty} \norm{f} \lesssim  t^{-\gamma} \norm{f}\;,
\end{align*}
which proves the results for $S_\eps(t)$. The results for $S_\eps^*(t)$ follow from the dual representation 
$\norm{S^*_\eps(t) f } = \sup_{\norm{g}=1} |\innerprod{f,S_\eps(t)g}|$.
\end{proof}
We now recall some standard estimates on the adjoint of the semigroup $S(t)$ generated by $\gen$. 
\begin{lemma}\label{lem:bounded}\label{l.semigroupstar}
Let $S^*(t)$ denote the adjoint of $S(t)$. For any $t\in (0,T]$, we have that 
\begin{align}
\norm{S^*(t) } & \leq  C_T\label{e:bounded2}\;, \\
\norm{\del_x S^*(t)} &\lesssim |t|^{-1/2} \label{e:bounded3}\;.
\end{align}
Moreover, there exists $\omega > 0$ such that
\begin{equation}
\norm{S^*(t)\left(1-\rho(x)\right)}\lesssim \exp(-\omega t) \;.
\end{equation}
\end{lemma}
\begin{proof}
The first result follows from Lemma \ref{l.semigroup} with $\eps=1$. The second result follows if we can show that the interpolation spaces of $(1-\gen)$ are the same as the Sobolev spaces interpolated by $(1-\del_x^2)$. Firstly, one can find a change of variables $Q$ such that 
\[
Q \gen Q^{-1} = V(x)\del_x + \del_x^2  
\]
where $Q$ and its inverse are bounded from $H^s$ into itself for any $s$ and $V$ is bounded. This change of variables can be found in Lemma \ref{lem:interpolation}. Hence, the interpolation spaces of $(1-\gen)$ are the same as the interpolation spaces of $(-V(x)\del_x + 1 - \del_x^2)$. Furthermore, we have the following fact: if $L_0$ generates an analytic semigroup on $\mathcal{B}$ and has interpolation spaces $\mathcal{B}^0_{\gamma}$, then $B+L_0$ has the same interpolation spaces, whenever $B$ is a bounded operator from $\mathcal{B}^0_\gamma$ into $\mathcal{B}$, for some $\gamma \in [0,1)$ by \cite{hairer09}. It follows that $B + L_0 = (1-Q \gen Q^{-1})$ has the same interpolation spaces as $L_0 = (1-\del_x^2)$, which proves the claim. The third result follows using standard machinery from spectral theory, similar to those used in Lemma \ref{l.semigroup}. 
\end{proof}
Since it will not affect any of our future estimates, we will assume from this point on that $\omega=1$. 
Notice that the semigroup $\Seps(t)$ satisfies the following rescaling identity
\begin{equation}\label{e:semigroupscaling}
\Seps(t) f(x/\eps) = (S(t/\eps^2)f)(x/\eps) \;.
\end{equation}
One can therefore think of the semigroup as zooming in on the highly oscillatory parts, evolving them (according to the diffusion generated by $\gen$) to very large times, and then zooming back out. In particular, combining this identity with Lemma \ref{lem:bounded} gives
\begin{equation}\label{e:elliptic}
\left\| \Seps^*(t)\left(1-\rho(x/\eps)\right) \right\| \lesssim \exp(-\omega t/\eps^2)\;,
\end{equation}
which will prove useful in the sequel.

\subsection{Interpolation Results}\label{ss.interpolation}

In order to prove convergence results in particular Sobolev spaces, we need to know the smoothing properties of the semigroup $\Seps(t)$. Estimates from analytic semigroup theory tell us which interpolation spaces of $\gen_\eps$ the solutions will live in. We would therefore like to obtain some embedding result between these interpolation spaces and the usual Sobolev spaces. It would be futile to look for an embedding result uniformly in $\eps$, the best we can do is the following lemma, which, for a price, grants us the ability to switch back and forth between interpolation spaces and Sobolev spaces.

\begin{lemma}
One has the following two inequalities
\begin{align}
\norm{(1-\del_x^2)^{\gamma} f } &\lesssim  \eps^{-2\gamma} \norm{(1-\gen_\eps)^{\gamma}f}\label{e.interpolationlemma1}\\
\norm{(1-\gen_\eps)^{-\gamma} f } &\lesssim  \eps^{-2\gamma} \norm{(1-\del_x^2)^{-\gamma}f}\label{e.interpolationlemma2}
\end{align}
for any $\gamma \in [0,1]$ and any $f$ for which the two norms are finite. 
\label{lem:interpolation}
\end{lemma}
\begin{proof}
We start by proving the first inequality, the second will follow with a simple argument. To prove the first claim we apply the Cald\'eron-Lions interpolation theorem~\cite{reedsimon_fourier} to obtain a relationship between the interpolation spaces given by
\begin{align*}
\norm{\cdot}_X^{(0)} &= \norm{\cdot}\;, \qquad \norm{\cdot}_{X}^{(1)} = \norm{(1-\gen_\eps)\cdot}\;, \\
\norm{\cdot}_Y^{(0)} &= \norm{\cdot}\;, \qquad \norm{\cdot}_{Y}^{(1)} = \norm{(1-\del_x^2)\cdot}\;.
\end{align*}
It guarantees that, for the identity operator $I$, one has 
\begin{equation}
\norm{I}_{L(X^{(\gamma)},Y^{(\gamma)})} \leq \norm{I}_{L(X^{(0)},Y^{(0)})}^{1-\gamma}\norm{I}_{L(X^{(1)},Y^{(1)})}^{\gamma} \;,
\label{e:calderon1}
\end{equation}
where $X^{(\gamma)}$ and $Y^{(\gamma)}$ are the interpolation spaces given by completing $\Ltwopi$ with respect to the norms $\norm{(1-\gen_\eps)^{\gamma}\cdot}$ and $\norm{(1-\del_x^2)^{\gamma}\cdot}$ respectively. 
\\
It is clear that 
\[
\norm{I}_{L(X^{(0)},Y^{(0)})}=1\;,
\]
since this is just the norm of the identity operator in $\Ltwopi$. The first claim thus follows if we can show that
\[
\norm{I}_{L(X^{(1)},Y^{(1)})} \lesssim \eps^{-2}\;,
\] 
which is equivalent to proving that
\begin{equation}
\norm{(1-\del_x^2)(1-\gen_\eps)^{-1}f} \lesssim \eps^{-2} \norm{f}\;.
\label{e:calderon2}
\end{equation}
We will achieve this by simplifying the operator $\gen_\eps$ through two transformations. Firstly, for the generator $\gen$, one can easily find a change of variables $z=\phi(x)$ with inverse $x=\psi(z)$ such that 
\begin{equation}\label{e.driftgenerator}
\gen f (x) = \left( B(\psi(z))\del_z + \del_z^2 \right)(f\circ \psi) (z) \;,
\end{equation}
where $B = \sqrt{2}\frac{b}{\sigma} - \frac{1}{\sqrt{2}}\sigma'$ and $\phi$ solves the ordinary differential equation
\begin{equation}\label{e.phiODE}
\phi'(x) = \frac{1}{\sqrt{2}}\sigma (\phi(x)) \;,
\end{equation}
with boundary condition $\phi(0)=0$. Given this change of variables, it is easy to find the corresponding change of variables for $\gen_\eps$, in fact, if we set $z=\eps \phi(x/\eps)$ we have that
\begin{equation}\label{e.}
 \gen_\eps f(x) = \left( \frac{1}{\eps}B(\psi (z/\eps)) \del_z + \del_z^2 \right)(f\circ\psi_\eps)(z)\;,
\end{equation}
where $\psi_\eps(\cdot) = \eps \psi(\cdot/\eps)$. Secondly, we hope to make the operator self-adjoint. To do this, we weight our space using the invariant measure of the underlying generator. Let $g(y)$ be the invariant density for the generator $\left( \frac{\sqrt{2}B(y)}{\sigma(y)}\del_y + \del_y^2 \right)$. One can show that 
\begin{equation} \label{e.operatordecomposed}
\gen_\eps f(x) = g(x/\eps)^{-1/2} (\Aeps u) (\eps \phi(x/\eps) )\;,
\end{equation}
where $u = g(\psi(\cdot/\eps))^{1/2}f\circ\psi_\eps$. The Schr\"odinger operator $\Aeps$ is defined by 
\[
\Aeps u(z) = \frac{1}{\eps^2}W(\psi(z/\eps))u(z) + \del_z^2 u(z)
\]
where $W = g^{1/2} \left( \frac{\sqrt{2}B}{\sigma} \del_y + \del_y^2 \right) g ^{-1/2}$. 
 We then have that 
\begin{align*}
\norm{(1-\del_x^2)(1-\gen_\eps)^{-1}f(x)}& \leq \eps^{-2}\norm{(g^{-1/2})''}_{\infty}\norm{(1-\Aeps)^{-1} u (\eps \phi(x/\eps))}\\ &+ \eps^{-1}\norm{(g^{-1/2})'}_{\infty}\norm{\del_x(1-\Aeps)^{-1} u (\eps \phi(x/\eps))}\\ &+ \norm{g^{-1/2}}_{\infty}\norm{\del_x^2(1-\Aeps)^{-1} u (\eps \phi(x/\eps))}\;.
\end{align*}
One can easily deduce the boundedness of $g^{-1/2}$ and its derivatives from Assumption~\ref{ass:elliptic}. Moreover, we have that
\begin{align*}
\norm{\del_x (1-\Aeps)^{-1}u(\eps \phi(x/\eps))}^2 &= \norm{\left((1-\Aeps)^{-1}u)'(\eps \phi(x/\eps)\right)\phi'(x/\eps)}^2 \\
&=\frac{1}{2\pi}\int_0^{2\pi} |\left((1-\Aeps)^{-1}u)'(\eps \phi(x/\eps)\right)\phi'(x/\eps)|^2 dx\\
&=\frac{1}{2\pi}\int_0^{\eps \phi(2\pi/\eps)} |\del_z (1-\Aeps)^{-1} u (z) |^2 |\phi'(\psi(z/\eps))|dz\\
&\leq \norm{\phi'}_{\infty} \norm{\del_z (1-\Aeps)^{-1}u}_{\phi}^2\;,
\end{align*}
where $\norm{\cdot}_{\phi}$ denotes the usual $L^2$ norm but over the interval $[0,\eps \phi(2\pi/\eps)]$ as in the integral above. We can similarly show that
\begin{equs}
\norm{\del_x^2(1-\Aeps)^{-1}u(\eps\phi(x/\eps))} &\leq \norm{(\phi')^{3}}^{1/2}_{\infty}\norm{\del_z^2 (1-\Aeps)^{-1}u}_{\phi} \\ &\quad + \eps^{-1}\norm{\frac{(\phi'')^2}{{\phi}}}^{1/2}_{\infty}\norm{\del_z (1-\Aeps)^{-1}u}_{\phi}\;.
\end{equs}
We can deduce the boundedness of the above expressions involving $\phi$ using \eqref{e.phiODE} and Assumption \ref{ass.elliptic}. We therefore have the bound 
\begin{equs}
\norm{(1-\del_x^2)(1-\gen_\eps)^{-1}f} &\lesssim \eps^{-2}\norm{(1-\Aeps)^{-1}u}_{\phi} + \eps^{-1} \norm{\del_z(1-\Aeps)^{-1}u}_{\phi} \\
&\quad + \norm{\del_z^2 (1-\Aeps)^{-1}u}_{\phi} \;.
\end{equs}
We now claim the following bounds to hold, as operator norms from $L^2_\phi \to L^2_\phi$ in the sense of the norm defined above:
\begin{align}
\norm{(1-\Aeps)^{-1}}_\phi &\leq 1 \label{e.operatorbound1}\;, \\
\norm{\del_z^2 (1-\Aeps)^{-1}}_\phi &\lesssim \eps^{-2} \label{e.operatorbound2}\;.
\end{align}
Note that these bounds immediately imply $\norm{\del_z(1-\Aeps)^{-1}}_{\phi} \lesssim \eps^{-1}$ which follows from the Cauchy-Schwartz inequality. These three operator bounds are enough to prove \eqref{e:calderon2}, since by changing back to the $x$ variables, we have that 
\begin{align*}
\norm{u}_{\phi} &= \norm{g(x/\eps)^{1/2}f(x) (\phi'(x/\eps))^{1/2}} \leq \norm{g}^{1/2}_{\infty} \norm{\phi'}_{\infty}^{1/2} \norm{f}\;.
\end{align*}
Hence we need only prove the claimed bounds. To prove \eqref{e.operatorbound1}, we utilise the 
identity
\[
\rm{spec} (1- \Aeps ) = {\rm spec} \left(1-\gen_\eps \right)\;,
\]
which follows from the fact that $\Aeps$ and $\gen_\eps$ are conjugated via a bounded operator with bounded inverse.
Since $\gen_\eps$ generates a Markov semigroup, elements in its spectrum have positive real part. 
Since $(1-\Aeps)$ is self-adjoint in the Hilbert space generated by the norm $\norm{\cdot}_{\phi}$ with the corresponding inner product, it thus follows that
\[
\norm{(1-\Aeps)^{-1}}_\phi \leq 1
\] 
using the spectral theorem \cite{hairer09}. By writing $\del_z^2$ in terms of $\Aeps$ and $W$, we also have that 
\begin{equs}
\norm{\del_z^2 (1-\Aeps)^{-1}}_{\phi} &\leq 1 +  \norm{\left(1+\frac{1}{\eps^2}W(\psi(\cdot/\eps))\right)(1-\Aeps)^{-1}}_{\phi} \\
&\lesssim 1 + \eps^{-2} (1+\norm{W}_{\infty}) \norm{(1-\Aeps)^{-1}}_{\phi}\;,
\end{equs}
which proves \eqref{e.operatorbound2} and hence \eqref{e.interpolationlemma1}. 
To prove the second part of the lemma, just as in \eqref{e:calderon2} it is sufficient to show that
\[
\norm{(1-\gen_\eps)^{-1}(1-\del_x^2) f} \leq C \eps^{-2} \norm{f}\;.
\]
But we can use the fact that the operator norm is preserved under taking the adjoint, so that
\[
\norm{(1-\gen_\eps)^{-1}(1-\del_x^2) } = \norm{(1-\del_x^2)(1-\gen_\eps^*)^{-1}} \;.
\]
It is therefore sufficient to prove \eqref{e:calderon2} with $\gen_\eps$ replaced with its adjoint $\gen_\eps^*$. An easy calculation shows that
\[
\gen_\eps^* = {\Ltilde_\eps}  + \frac{1}{\eps^2} U(x/\eps) 
\]
where 
\[
\Ltilde_\eps = \frac{1}{\eps} \tilde{b}(x/\eps)\del_x  + \frac{1}{2}\sigma^2(x/\eps) \del_x^2\;.
\]
We can reduce $\Ltilde_\eps$ to a Schr{\"o}dinger operator with potential $\tildeW$ in the same way that we did for $\gen_\eps$, and hence reduce $\gen_\eps^*$ to a Schr{\"o}dinger operator with potential $\tildeW + U$. The second claim then follows similarly to the first. 
\end{proof}
\begin{rmk}\label{rmk.sharpness}
We would like to briefly comment on the sharpness of the two estimates obtained in Lemma \ref{lem:interpolation}. The second estimate \eqref{e.interpolationlemma2} is sharp. In fact, in the case $\sigma=1$, upon rewriting the estimate in the adjoint setting, as done in the proof, it is clear that taking $f=\rho(x/\eps)$ will prove sharpness. Unfortunately, this argument does not work for the first estimate \eqref{e.interpolationlemma1}. This comes down to the unlucky fact that the zero eigenvector of $\gen_\eps$ is the constant function (and not $\rho(x/\eps)$), which of course does not yield powers of $\eps$ when integrated. In fact, we believe that estimate \eqref{e.interpolationlemma1} is not sharp. However, improving the estimate would not considerably improve the strength of results in the sequel, so we do not attempt to do so.       
\end{rmk} 

\subsection{Estimating the semigroup}\label{ss.estimating}
A key ingredient in proving all three convergence results is an estimate on the low Fourier modes of the mild solution to \eqref{e:formulationSPDE}, that is
\[
\innerprod{u_\eps(t),e_m} = \sum_k \int_0^t \innerprod{\Seps(t-s) q_k^\eps e_k,e_m}dW_k(s)\;,
\]
for $|m| < \eps^{-1}$, recalling the notation $q_k^\eps(x) = q_k(x/\eps)$. This could be achieved by estimating $\Seps(t-s) q_k^\eps e_k$. However, this becomes troublesome when $k$ is large. It is more convenient to exploit the fact that
\[
\innerprod{u_\eps(t),e_m} = \sum_k \int_0^t \innerprod{ q_k^\eps e_k, \Seps^*(t-s)e_m}dW_k(s)
\]
and estimate $\Seps^*(t-s)e_m$, with $m$ fixed. We will prove that
\[
S^*(t)e_m(x) \approx \rho(x/\eps)e_m(x) e^{-\mu m^2 t} + f_\eps^\BL (x,t)\;,
\]
uniformly in $t\in[0,T]$. As before, $\rho$ is the invariant density of the $\gen$ and we define the ``boundary layer'' $f_\eps^\BL $ as a term that corrects the approximation when $t=\bigoh(\eps^2)$ and converges rapidly to zero when $t > \eps^2$. Such results can be obtained in the setting of martingale problems \cite{papa77} however, as we would like to obtain a bit of control over rates of convergence, we take the approach used in \cite{bensoussan78,pavliotis08}. 
\\

Let us set $f_\eps(x,t)=\Seps^*(t)e_m(x)$. We would then like to find an approximate solution to the PDE
\begin{equation}\label{e:PDEf}
\del_t f_\eps(x,t) = \gen_\eps^* f_\eps(x,t)  \;,\quad f_\eps(x,0) = e_m(x)\;,
\end{equation}
where the adjoint generator $\gen^*_\eps$ has periodic boundary conditions on $[0,2\pi]$. The standard approach to problems of this kind is to rewrite \eqref{e:PDEf} in the new variables $\tildex = x$ and $\tildey = x/\eps$ and separate the macroscopic dynamics from the microscopic dynamics. One can then obtain an approximate solution by introducing a power series expansion
\[
f_\eps(\tildex,\tildey,t) = f_0(\tildex,\tildey,t) + \eps f_1(\tildex,\tildey,t) + \eps^2 f_2(\tildex,\tildey,t)+ \dots
\]
into the PDE \eqref{e:PDEf} and solving for $f_0, f_1, f_2$ by matching powers of $\eps$. Under this procedure, one obtains 
\begin{align*}
f_0(x,x/\eps,t) &= \rho(x/\eps)e_m(x)e^{-\mu m^2 t}\;,\\
f_1(x,x/\eps,t) &= \Phi_1(x/\eps)\del_x e_m(x) e^{-\mu m^2 t}\;,\\
f_2(x,x/\eps,t) &= \Phi_2(x/\eps)\del_x^2 e_m(x) e^{-\mu m^2 t}\;,
\end{align*}
where $\Phi_1, \Phi_2 \in \mathcal{C}_b^2$. This approach encounters a small problem in that the approximation breaks down when $t=\bigoh(\eps^2)$. The problem is averted by introducing a temporal boundary layer term, also known as a corrector, which we define as
\[
f_{\eps}^\BL (x,t) = \left(\Seps^*(t)\left(1-\rho(x/\eps)\right)\right)e_m(x)\;.
\]
One can see that the boundary layer term corrects the discrepancy in the initial condition of the approximation $S_\eps^*(t)e_m(x) \approx \rho(x/\eps)e_m(x) e^{-\mu m^2 t}$, indeed, the boundary layer term's sole purpose is to correct the approximation for small times $t$. We therefore define the remainder term $r_\eps$ by setting
\begin{equation}
f_\eps(x,t) = f_0(\tildex,\tildey,t) + \eps f_1(\tildex,\tildey,t) + \eps^2 f_2(\tildex,\tildey,t)
+ f_\eps^\BL (x,t) + r_\eps(x,t) \label{e:remainderexplicit} 
\end{equation}
Note that our definition of the remainder depends explicitly on the wavenumber $m$, however, for convenience we omit this from the notation. Using the method described above, one can write down the following convenient expression for the remainder.
\begin{lemma}\label{lem:remainderPDE}\label{l.representation}
If $\eps|m| < 1$ and $r_\eps$ is the remainder defined in \eqref{e:remainderexplicit} then we can write 
\begin{equs}
r_\eps(x,t) =\; & \Seps^*(t)r_\eps(x,0)+  \;\eps \int_0^t \Seps^*(t-s) F_1(x,x/\eps,s) \,ds\label{e:remainderduhamel}\\ 
 +&\; \eps^2 \int_0^t \Seps^*(t-s) F_2(x,x/\eps,s) ds + \int_0^t \Seps^*(t-s)(\del_s -\gen_\eps^*)f_\eps^\BL (x,s)\,ds \;,
\end{equs}
where the functions $F_1$ and $F_2$ satisfy the bounds
\begin{equation}
\norm{F_1(t)} \lesssim (1\vee |m|^3) e^{-\mu m^2 t}
\quad \text{and} \quad
\norm{F_2(t)} \lesssim (1\vee |m|^4) e^{-\mu m^2 t}\;,
\end{equation}
where $\mu > 0$ is a constant determined by $\gen$. 
\end{lemma}
\begin{proof}
The method of proof is described above. One can find similar calculations in \cite{bensoussan78, pavliotis08}. 
\end{proof}
Each term in \eqref{e:remainderduhamel} can be bounded without too much trouble, except for the boundary layer term, which we shall treat separately. 
\begin{lemma}\label{lem:boundarybound}\label{l.boundarylayer}
If $\eps|m|<1$, then for any $t\in[0,T]$, we have that 
\begin{equation}
\norm{f_\eps^\BL (t)} \lesssim \exp(- t/\eps^2 )\;.
\end{equation}
Furthermore, for any $s\in[0,t]$ we have that
\begin{equation}
 \norm{ S_\eps^*(t-s)(\gen_\eps^* - \del_s) f_{\eps}^\BL (x,s)} \lesssim \frac{m}{\eps}\exp({-s/\eps^2})\;.
\end{equation}
In both cases, the proprtionality constants are independent of $m$, provided that $\eps |m| \le 1$.
\end{lemma}
\begin{proof}
For the sake of brevity, throughout this proof and the next we will simply write $m$ instead of $1\vee |m|$. We also introduce the shorthand
\[
\rhobar_{t/\eps^2}(x/\eps) := \bigl(S^*(t/\eps^2)\left(1-\rho\right)\bigr)(x/\eps) = \bigl(S_\eps^*(t)\left(1-\rho^\eps\right)\bigr)(x)
\]
where the last identity follows from the rescaling property \eqref{e:semigroupscaling}, recalling that $\rho^\eps(x) = \rho(x/\eps)$. We then have that 
\begin{align*}
\norm{f_\eps^\BL (t)} = \norm{\rhobar_{t/\eps^2}^\eps e_m} =\norm{\rhobar_{t/\eps^2}^\eps} \lesssim \exp(-t/\eps^2)\;, 
\end{align*}
which follows from \eqref{e:elliptic}. For the second result, notice that
\begin{align*}
(\gen_\eps^* - \del_s) f_\eps^\BL (x,s) = &-\frac{1}{\eps} b(x/\eps)\rhobar_{s\over\eps^2}(x/\eps)\del_x e_m(x) + \del_x \left( \sigma^2(x/\eps)\rhobar_{s\over\eps^2}(x/\eps)\right)\del_x e_m(x)\\ &+ \frac{1}{2} \sigma^2 (x/\eps) \rhobar_{s\over\eps^2}(x/\eps)\del_x^2 e_m(x)\;. 
\end{align*}
Therefore, the quantity
\begin{align*}
\norm{S_\eps(t-s) (\gen_\eps^* - \del_s) f_\eps^\BL (x,s)  } \lesssim \norm{(\gen_\eps^* - \del_s) f_\eps^\BL (x,s)}
\end{align*}
is bounded by
\begin{equation}\label{e.boundarylayerterm}
\frac{m}{\eps} \norm{b \rhobar_{s/\eps^2}} + \frac{m}{\eps^2} \norm{\del_x \left( \sigma^2 \rhobar_{s/\eps^2} \right)} + m^2 \norm{\sigma^2 \rhobar_{s/\eps^2}}\;.
\end{equation}
We furthermore have the bound
\begin{align*}
\norm{\del_x \left( \sigma^2 \rhobar_{s/\eps^2} \right)} &\lesssim \norm{\del_x \sigma^2}_{\infty} \norm{\rhobar_{s/\eps^2}} + \norm{\sigma^2}_{\infty} \norm{\del_x \rhobar_{s/\eps^2}}\\
 &\lesssim \left(\norm{\del_x\sigma^2}_{\infty} + \norm{\sigma^2}_{\infty}  \right) \exp(-s/\eps^2)\;,
\end{align*}
where we have used the bound 
\begin{equation}\label{e.rhobarclaim}
\norm{\del_x \rhobar_{s/\eps^2}} \lesssim \exp (-s/\eps^2)\;.
\end{equation}
which we will prove shortly. Therefore, we can bound \eqref{e.boundarylayerterm} by
\begin{align*}
\frac{m}{\eps}\norm{b}_\infty \exp(-s/\eps^2) &+ \frac{m}{\eps}(\norm{\del_x\sigma^2}_{\infty} + \norm{\sigma^2}_{\infty}) \exp(-s/\eps^2) + m^2 \norm{\sigma^2}_\infty \exp(-s/\eps^2)\\ 
&\lesssim \frac{m}{\eps}\exp(-s/\eps^2)\;.
\end{align*}
Here we have used Assumption \ref{ass:elliptic} to obtain the required bounds on $b$ and $\sigma$ and also the assumption $\eps|m|<1$. This proves the bounds stated in the lemma. To prove the claimed bound \eqref{e.rhobarclaim}, first assume $s>\eps^2$, then
\begin{align*}
\norm{\del_x \rhobar_{s/\eps^2}} &= \norm{\del_x S^*(1) S^*(s/\eps^2-1)(1-\rho)}\\
&\lesssim \norm{\del_x S^*(1)} \norm{S^*(s/\eps^2 -1)(1-\rho)} \lesssim \exp(-s/\eps^2)\;,
\end{align*}
where we have used Lemma \ref{l.semigroupstar}. If $s\leq \eps^2$ then 
\begin{align*}
\norm{\del_x \rhobar_{s/\eps^2}} &= \norm{\del_x (\gen^*)^{-1} S^*(s/\eps^2) \gen^*(1-\rho)}\\ &\lesssim \norm{\del_x (\gen^*)^{-1}} \norm{\gen^*1} < \infty\;.
\end{align*}
The boundedness of $\norm{\del_x (\gen^*)^{-1}}$ follows from the proof of Lemma \ref{l.semigroupstar}, where we showed that $\gen$ and $\del_x^2$ share the same interpolation spaces. We can therefore bound $\norm{\del_x \rhobar_{s/\eps^2}}$ uniformly for $s\in[0,\eps^2]$, which, together with the bound for $s>\eps^2$, implies \eqref{e.rhobarclaim}.
\end{proof}
Note that $r_\eps$ contains extra terms $f_1$ and $f_2$ that are only in place to facilitate the proof of Lemma \ref{l.representation}. We therefore define the following new remainder for the approximation that we actually use
\[
\Seps^*(t)e_m(x) = \rho(x/\eps) e_m(x)e^{ -\mu m^2 t} + f_\eps^\BL (x,t) + R_\eps(x,t)\;.
\]
We now obtain the estimates on $R_\eps$.
\begin{lemma}
\label{l.remainder}
If $\eps|m|<1$ then we have the estimates
\begin{align}
\suptime \norm{R_\eps(t)} \lesssim \eps (1\vee|m|) \quad\text{and} \quad \int_0^T\norm{R_\eps(t)}_{H^1} dt \lesssim (1\vee|m|) \;.
\end{align}
We also have that
\begin{equation}
\suptime \norm{\del_t R_\eps(t)} \lesssim \frac{1\vee |m|^2}{\eps^2}\;.
\end{equation}
 
\end{lemma}

\begin{proof}
We will first prove the bound for $\norm{R_\eps(t)}$. From the definition of the remainder $R_\eps$, we have that 
\[
R_\eps(t) = r_\eps(t) +  \eps f_1(t) +  \eps^2 f_2(t)\;\;,  
\]
where $f_1(t) = i m \Phi_1(x/\eps) e_m(x)e^{-\mu m^2 t}$ and $f_2(t) = - m^2 \Phi_2(x/\eps) e_m(x) e^{-\mu m^2 t}$. As a consequence, we obtain 
\[
\norm{R_\eps(t)} \lesssim \norm{r_\eps(t)} + \eps \norm{f_1(t)} + \eps^2 \norm{f_2(t)} \lesssim \norm{r_\eps(t)} + \eps m \;. 
\]
From Lemma \ref{l.representation} we have that 
\begin{align*}
\norm{r_\eps(t)} &\lesssim \norm{\Seps^*(t)r_\eps(0)} + \eps \int_0^t \norm{\Seps^*(t-r)F_1(r)}dr\\ &+ \eps^2 \int_0^t \norm{\Seps^*(t-r)F_2(r)} dr + \int_0^t \norm{ \Seps^*(t-r) (\del_r - \gen_\eps^*)f_\eps^\BL (r)}dr  \;.
\end{align*}
Each of the above terms shall now be bounded separately. Using the uniform boundedness of the semigroup, we have that 
\[
\norm{\Seps^*(t)r_\eps(0)} \lesssim \norm{r_\eps(0)} \lesssim \eps m\;,
\]
which follows from \eqref{e:remainderexplicit}. If we use the bound on $\norm{F_1}$ given in Lemma \ref{l.representation} we have that 
\begin{align*}
\eps \int_0^t \norm{\Seps^*(t-r)F_1(r)}dr &\lesssim \eps \int_0^t \norm{F_1(r)}dr\notag \\ 
&\lesssim 
\eps \int_0^t m^3 \exp(-\mu m^2 r) dr \lesssim \eps m\;.
\end{align*}
Similarly, we have that 
\[
\eps^2 \int_0^t \norm{\Seps^*(t-r)F_2(r)}dr\lesssim \eps^2 m^2 \lesssim \eps m\;.
\]
Finally, from Lemma \ref{l.boundarylayer} we have that 
\[
\int_0^t \norm{ \Seps^*(t-r) (\del_r - \gen_\eps^*)f_\eps^\BL (r)}dr \lesssim \int_0^t \frac{m}{\eps}e^{-r/\eps^2}dr  \lesssim \eps m \;.
\]
Putting all this together, we have that 
\[
\norm{R_\eps(t)} \lesssim \eps m \;,
\]
whenever $\eps |m| < 1$. We now seek the bound on $\norm{R_\eps(t)}_{H^1}$. We have that 
\begin{align*}
\norm{R_\eps(t)}_{H^1} &\lesssim \norm{r_\eps(t)}_{H^1} + \eps \norm{f_1(t)}_{H^1} + \eps^2\norm{f_2(t)}_{H^1}\\ 
&\lesssim \norm{(1-\del_x^2)^{1/2}(1-\gen_\eps)^{-1/2}} \norm{(1-\gen_\eps)^{1/2}r_\eps(t)} + m + \eps m^2\\
&\lesssim \eps^{-1} \norm{(1-\gen_\eps)^{1/2}r_\eps(t)} + m \;.
\end{align*}
Here we have used Lemma \ref{lem:interpolation} to switch between the the $\gen_\eps$ and $\del_x^2$ interpolation spaces. We have from Lemma \ref{l.representation} that 
\begin{equs}
\norm{(1-\gen_\eps)^{1/2}r_\eps(t)} &\lesssim  \norm{\Seps^*(t)(1-\gen_\eps)^{1/2}r_\eps(0)} + \eps \int_0^t \norm{\Seps^*(t-r)(1-\gen_\eps)^{1/2}F_1(r)}dr\\ &\quad + \eps^2 \int_0^t \norm{\Seps^*(t-r)(1-\gen_\eps)^{1/2}F_2(r)} dr\\ &\quad + \int_0^t \norm{ \Seps^*(t-r)(1-\gen_\eps)^{1/2} (\del_r - \gen_\eps^*)f_\eps^\BL (r)}dr \;.
\end{equs}
From Lemma \ref{l.semigroupstar}, we have that 
\[
\norm{\Seps^*(t)(1-\gen_\eps)^{1/2}} \lesssim |t|^{-1/2}\;,
\] 
for any $t\in(0,T]$. Therefore, we have that 
\[
\norm{\Seps^*(t)(1-\gen_\eps)^{1/2}r_\eps(0)} \lesssim |t|^{-1/2} \norm{r_\eps(0)} \lesssim \eps m |t|^{-1/2}\;.
\] 
Furthermore, we have that 
\begin{align*}
\eps \int_0^t \norm{\Seps^*(t-r)(1-\gen_\eps)^{1/2}F_1(r)}dr &\lesssim \eps \int_0^t |t-r|^{-1/2} \norm{F_1(r)}dr\notag \\ 
&\lesssim 
\eps \int_0^t m^3 |t-r|^{-1/2} \exp(-\mu m^2 r) dr\\ 
&\lesssim \eps m \left({|t|^{-1/2}} + m^2\exp(-\mu m^2 t)\right) \;. 
\end{align*}
Here we have bounded the above integral by splitting the range of integration in half. Similarly, we have that 
\begin{align*}
\eps^2 \int_0^t \norm{\Seps^*(t-r)(1-\gen_\eps)^{1/2}F_2(r)}dr \lesssim \eps m \left({|t|^{-1/2}} + m^2\exp(-\mu m^2 t)\right)\;.  
\end{align*}
Finally, from Lemma \ref{l.boundarylayer} we have that
\begin{align*}
\int_0^t \norm{\Seps^*(t-r)&(1-\gen_\eps)^{1/2}(\del_r - \gen_\eps^*)f_\eps^\BL (r)} dr \\
&\lesssim
\frac{m}{\eps}  \int_0^t |t-r|^{-1/2} \exp(- r/\eps^2) dr\\  &\lesssim \eps m \left( |t|^{-1/2} + \frac{\exp(-t/\eps^2)}{\eps^2}\right) \; . 
\end{align*}
Putting this all together, along with the fact that $\eps|m|<1$, we have the bound
\[
\norm{R_\eps(t)}_{H^1} \lesssim  m \left( 1 +  |t|^{-\gamma/2} + m^2 \exp(-\mu m^2 t) + \frac{\exp(-t/\eps^2)}{\eps^2} \right)\;,
\]
and the requested bound on $\int_0^T\norm{R_\eps(t)}_{H^1}dt $ follows. For the final estimate, we use the definition
\begin{align*}
 R_\eps(t) = \Seps^*(t) e_m(x) -\rho(x/\eps)e_m(x)e^{-\mu m^2 t} - \rhobar_{t/\eps^2}(x/\eps)e_m(x)\;.
\end{align*}
We then have
\begin{align*}
\suptime\norm{\del_t R_\eps(t) } &\lesssim \suptime\norm{\del_t \Seps^*(t) e_m} + m^2\norm{\rho} + \suptime\frac{\norm{\del_t \rhobar}}{\eps^2} \\
&\lesssim \suptime\norm{\del_t \Seps^*(t) e_m} + \frac{m^2}{\eps^2}\;,
\end{align*}
since the boundedness of $\suptime\norm{\del_t \rhobar}$ and $\norm{\rho}$ are guaranteed by the smoothness of $b$ and $\sigma$. Due to the uniform boundedness of the semigroup $\Seps(t)$, we also have that
\[
\suptime\norm{\del_t \Seps^*(t)e_m} = \suptime\norm{\Seps^*(t) \gen_\eps^* e_m} \lesssim \norm{ \gen_\eps^* e_m} \lesssim \frac{m^2}{\eps^2}\;,
\] 
where the last inequality follows from the smoothness assumptions placed on $b$ and $\sigma$. This proves the result. 
\end{proof}

\section{Convergence results}\label{s:convergence}
In this section, we shall state the precise formulation of the main results and then provide their proofs in full detail. The first convergence result is as follows. 
\begin{thm}\label{thm:stronglimit}\label{thm.stronglimit}
Suppose $u_\eps$ satisfies \eqref{e:formulationSPDE} and the conditions given in Assumptions \ref{ass:elliptic}, \ref{ass.strongnoise} hold true. Suppose furthermore that $u$ solves the stochastic heat equation
\begin{equation}\label{e.stronglimitSPDE}
du(x,t) = \mu \del_x^2 u(x,t) dt + \sum_k \innerprod{q_k,\rho}e_k(x)dW_k(t)\;,
\end{equation}
with $u(x,0)=0$. Let $s_\alpha = 0\vee {3\over 2}(1-2\alpha)$, then for any $s>s_\alpha$ there exists $\theta_0(s) > 0$ such that  
\begin{equation}
\label{e.stronglimit}
\E \suptime \norm{u_\eps(t)-u(t)}_{H^{-s}}^2 \lesssim \eps^{\theta} \;,
\end{equation}
for any $\theta<\theta_0(s)$. 
\end{thm}
\begin{rmk}
For the interested reader, the rate of decay $\theta_0$ given by our proof is
\[
\theta_0(s) = 2\alpha\wedge \frac{4}{3}(s-s_\alpha) \;.
\]
\end{rmk}
As stated in the introduction, the next theorem deals with the second order term of the solution $u_\eps$, obtained by subtracting the first order term (or in our case, setting $\innerprod{q_k,\rho}=0$) and scaling the noise up by some inverse factor of $\eps$. We have the following result. 
\begin{thm}\label{thm:weaklimit}
Suppose $u_\eps$ satisfies \eqref{e:formulationSPDE} with $\innerprod{q_k,\rho}=0$ for all $k \in \integers$ and the conditions given in Assumptions \ref{ass:elliptic}, \ref{ass.weaknoise} hold true for a given $\alpha \in (0,1)$. 

Then, there exists a probability space with a sequence of Wiener processes $\{\tildeW_k\}$
and processes $\{\hat u_\eps\}$ that are equal in law to $\{u_\eps\}$,
such that
\begin{equation}
\label{e.weaklimit}
\lim_{\eps \to 0} \E \suptime \norm{\eps^{-\alpha}\tildeu_\eps(t) - v(t)}_{H^{-s}}^2 =0\;,
\end{equation}
where $v$ is the solution to
\begin{equation}\label{e.weaklimitSPDE}
dv(x,t) = \mu \del_x^2 v(x,t) dt + \norm{\qbar \rho}_{-\alpha}\sum_k e_k(x) d\tildeW_k(t)\;,  
\end{equation}
with $v(x,0)=0$.
The convergence \eref{e.weaklimit} holds for any $s>{3\over 2}\bigl(\alpha \vee (1-\alpha)\bigr)$.
\end{thm}
The two preceding theorems always require some decay on the coefficients $q_k$, in particular the results do no treat SPDEs driven by space-time white noise, where $q_k=1$ for each $k\in \integers$. We know that in the space-time white noise case, the solution converges to the so-called wrong limit. The following result generalises this phenomena to a broad class of driving noise processes. 
\begin{thm}\label{thm.STWN}
Suppose $u_\eps$ satisfies \eqref{e:formulationSPDE} and that the conditions given in Assumption \ref{ass.elliptic}, \ref{ass.STWN} hold true. 
Then, there exists a probability space with a sequence of Wiener processes $\{\tildeW_k\}$
and processes $\{\hat u_\eps\}$ that are equal in law to $\{u_\eps\}$,
such that
\begin{equation}
\label{e.STWNlimit}
\lim_{\eps \to 0} \E \suptime \norm{\hatu_\eps(t) - \hatu(t)}_{H^{-s}}^2 =0\;,
\end{equation}
where $\hatu$ satisfies the stochastic heat equation
\begin{equation}\label{e.STWNlimitSPDE}
d\hatu(x,t) = \mu \del_x^2 \hatu(x,t) dt + \sum_{k} (|\innerprod{q_k,\rho}|^2 - |\innerprod{\qbar,\rho}|^2 + \norm{\qbar \rho}^2)^{1/2} e_k(x) d\hatW_k(t)\;,
\end{equation}
with $\hatu(x,0) = 0$.
The convergence \eref{e.STWNlimit} holds for any $s>s_\eta$, where 
\[
s_\eta = 
\begin{cases}
1 , 	&\text{if $\eta \in [0,1/2]$ }\;,\\
\frac{3}{2}(2-\eta)^{-1} , 	&\text{if $\eta \in [1/2,1)$ }\;.\\
\end{cases}
\]
(Here, $\eta$ is the constant appearing in Assumption~\ref{ass.STWN}.)
\end{thm}

\begin{rmk}
If one assumes that the driving noise does not depend on $\eps$, as is the case for space-time white noise, then the assumptions can be loosened. In particular, one can easily modify the proof of Theorem \ref{thm:stronglimit} to show the following. Suppose $u_\eps$ satisfies \eqref{e:formulationSPDE} with $q_k$ constants and that $u$ satisfies \eqref{e.stronglimitSPDE}, then 
\[
\lim_{\eps \to 0} \E \suptime \norm{u_\eps(t) - u(t)}_{H^{-s}}^2 =0\;,
\]
for $s$ large enough. Hence, we can still prove the limit, but at the expense of the rate of convergence. A similar result holds for \ref{thm.STWN}, in the case of constant $q_k$, in that we can weaken the assumption to just $q_k \to \qbar$, and still prove the limit \eqref{e.STWNlimit}. 
\end{rmk}

One might ask what happens if we approximate the noise by a smoother infinite dimensional Gaussian process, say $W_{\eps}$, which, for nonzero $\eps$ falls into the class of the classical (unsurprising) case, but as $\eps$ tends to zero, approaches something as irregular as space-time white noise, for instance. To this end, let $\varphi$ be a smooth test function on $\reals$ with compact support and $\varphi(0)=1$. We define the \emph{smoothened version} of \eqref{e:formulationSPDE} by
\[
du_{\eps}(t) = \gen_{\eps} u_{\eps}(t) dt + \sum_{k} \varphi(\eps k) q_{k}(x/\eps)e_{k}dW_{k}(t) \;.
\]
This smoothening procedure consists in taking the convolution of the noise with a scaled version of
the function $\tilde{\varphi}$, where $\tilde{\varphi}$ is the inverse Fourier transform of $\varphi$. The following corollary illustrates the transition between the classical case and the unsurprising case. 
\begin{corr}\label{corr:convolution}
Suppose $u_\eps$ satisfies the smoothened version of \eqref{e:formulationSPDE}, as defined above and that Assumptions \ref{ass.elliptic}, \ref{ass.STWN} hold true. Suppose furthermore that 
\begin{equation}\label{e:smoothlimit}
d\hatu(t) = \mu \del_{x}^{2} \hatu(t) dt + \sum_{k} \left( |\innerprod{q_k,\rho}|^2 - |\innerprod{\qbar,\rho}|^2 + \norm{(\qbar \rho) \star \tilde{\varphi}}^2 \right)^{1/2}  e_{k}d\hatW_{k}(t)
\end{equation}
Then $u_{\eps} \to \hat u$ in precisely the same sense as claimed in Theorem \ref{thm.STWN}.
\end{corr}

\begin{remark}
If we take $\tilde \varphi = 1$, then we recover Theorem~\ref{thm.stronglimit}. If on the other hand, we take $\varphi = 1$ (so that $\tilde \vphi = \delta$), then
we recover Theorem~\ref{thm.STWN}, so that we can view this corollary as an interpolation between the two theorems.
\end{remark}

The proof of Corollary~\ref{corr:convolution} is given on page~\pageref{e:corr1} below.
Before proving these results, we need a few specialised lemmas. The first technical lemma that we require will essentially provide us with a bound on the norm of the multiplication operator from $H^{-s}$ to $H^{-s}$, where the multiplier function is highly oscillatory. 
\begin{lemma}\label{lem:qksobolev}
For any $f \in H^1$ we have that 
\begin{equation}
\norm{(1-\del_x^2)^{-s/2}f^\eps(1-\del_x^2)^{s/2}}_{L^2 \to L^2} \lesssim \eps^{-s} \norm{f}_{H^1}\;,
\end{equation}
where $f^\eps(x)=f(x/\eps)$ denotes the corresponding multiplication operator.
\end{lemma}
\begin{proof}
We will equivalently prove that 
\[
\norm{f^\eps u}_{H^{-s}} \lesssim \eps^{-s} \norm{u}_{H^{-s}} \norm{f}_{H^1}\;, 
\]
this is done once more using Cald\'eron-Lions interpolation theorem \cite{reedsimon_fourier}. For $s=0$, the claim holds simply because
\[
\norm{f^\eps u} \lesssim \norm{f}_{\Linfty} \norm{u} \lesssim \norm{f}_{H^1} \norm{u}\;,
\]
which follows from a standard Sobolev embedding. For $s=1$ we also have the simple result for negative Sobolev norms
\[
\norm{f^\eps u}_{H^{-1}} \leq \norm{f^\eps}_{H^1} \norm{u}_{H^{-1}}\lesssim \frac{1}{\eps} \norm{f}_{H^1} \norm{u}_{H^{-1}} \;.
\]
The Cald\'eron-Lions theorem then implies that the multiplication operator has norm
\[
\norm{f^\eps}_{H^{-s}\to H^{-s}} \lesssim ( \norm{f}_{H^1} )^{1-s} (\frac{1}{\eps} \norm{f}_{H^1} )^{s} = \eps^{-s} \norm{f}_{H^1}\;,
\]
which proves the lemma. 
\end{proof}
In the next lemma, we obtain a control on the variance of the Gaussian process $u_\eps$ in the space of continuous functions taking values in $\Ltwopi$. This will be useful in deciding which Sobolev spaces contain the solutions uniformly in $\eps$ and hence determining where convergence occurs. 
\begin{lemma} \label{lem:apriori}\label{l.apriori}
Suppose $u_\eps$ satisfies \eqref{e:formulationSPDE} and the conditions given in Assumptions \ref{ass:elliptic}, \ref{ass.strongnoise} hold true. If $\alpha \in (1/2,1)$ then we have that
\begin{equation}\label{e:apriori1}
\E \suptime \norm{u_\eps(t) }^2 \leq C_T\;.
\end{equation}
Otherwise, if $\alpha \in (0,1/2]$ we have that
\begin{equation}\label{e:apriori2}
\E \suptime \norm{u_\eps(t) }^2 \lesssim \eps^{4\alpha-2-\delta} \;,
\end{equation}
for any $\delta \in (0,2)$. 
\end{lemma}
\begin{proof}
We utilise the fact that the semigroup $S_\eps(t)$ is a contraction semigroup when the domain is taken to be $L^2(\rho_\eps)$ with the corresponding norm and inner product, as introduced in Lemma \ref{l.semigroup}. This follows from the fact that the generator $\gen_\eps$ is self-adjoint in this weighted space combined with the fact that the generator has non-positive spectrum. One can therefore apply standard martingale-type inequalities for stochastic convolutions \cite{daprato92} to obtain
\begin{align*}
\E \suptime \norm{u_\eps(t)}_{\rho_\eps}^2 = \E \suptime& \Bigl\|\int_0^t S_\eps(t-s)Q_\eps dW(s)\Bigr\|_{\rho_\eps}^2 \notag \\
\lesssim \int_0^T& \norm{S_\eps(t)Q_\eps}^2_{HS,\rho_{\eps}}dt\;,
\end{align*}
where $\norm{\cdot}_{HS,\rho_\eps}$ denotes the Hilbert-Schmidt norm for operators mapping $L^2(\rho_\eps)$ into itself. We have already seen in Lemma \ref{l.semigroup} that the norms $\norm{\cdot}$ and $\norm{\cdot}_{\rho_\eps}$ are equivalent with their ratios bounded uniformly in $\eps \in (0,1)$. One can easily show that the same is true for the Hilbert-Schmidt norms $\norm{\cdot}_\HS$ and $\norm{\cdot}_{HS,\rho_\eps}$. 
Hence we have that
\begin{equation}\label{e.HS}
\E \suptime \norm{u_\eps(t)}^2 \lesssim \int_0^T \norm{S_\eps(t)Q_\eps}^2_\HS dt\;.
\end{equation}
Since $\alpha \in (1/2,1)$ implies that the noise is Hilbert-Schimdt, the result \eqref{e:apriori1} follows immediately from \eqref{e.HS}. Now suppose $\alpha \in (0,1/2]$, then
\begin{align*}
\norm{S_\eps(t)Q_\eps}_\HS^2 &= \sum_{k\in\integers} \norm{S_\eps(t)q_k^\eps e_k}^2\\
 &\lesssim \sum_{k\in\integers} (1\wedge|k|^{-2\alpha}) \norm{S_\eps(t)\qbar_k^\eps e_k}^2\;,
\end{align*}
where $\qbar_k = q_k/\norm{q_k}$ and $\qbar_k^\eps = \qbar_k(\cdot/\eps)$. However, we can trade the smoothness of the $\qbar_k$ to obtain a little more decay as $k$ gets large. In particular, we can write
\begin{align*}
\norm{S_\eps(t)\qbar_k^\eps e_k}^2 = (1+k^2)^{-\nu} \norm{S_\eps(t)\qbar_k^\eps(1-\del_x^2)^{\nu/2}e_k}^2\;,
\end{align*}
and using estimates from Lemmas \ref{lem:interpolation} and \ref{lem:qksobolev} we have that 
\begin{equs}
\norm{S_\eps(t)\qbar_k^\eps(1-\del_x^2)^{\nu/2}e_k}^2 &\leq \norm{S_\eps(t) (1-\gen_\eps)^{\nu/2}}^2\norm{(1-\gen_\eps)^{-\nu/2}(1-\del_x^2)^{\nu/2}}^2\\
&\qquad \times \norm{(1-\del_x^2)^{-\nu/2}\qbar_k^\eps(1-\del_x^2)^{\nu/2}e_k}^2\\
& \lesssim (t^{-\nu}) ( \eps^{-2\nu} ) (\eps^{-2\nu} \norm{\qbar_k}_{H^1}^2) \;.
\end{equs}
Therefore, we have that 
\begin{align*}
\norm{S_\eps(t)Q_\eps}_\HS \lesssim \eps^{-2\nu} t^{-\nu/2} \left( \sum_{k\in\integers} (1\wedge |k|^{-2\alpha - 2\nu})\norm{\qbar_k}_{H^1}^2 \right)^{1/2}\;,
\end{align*}
for any $\nu\in [0,1)$. If we set $\nu = 1/2 - \alpha +\delta$ then, given the uniform boundedness of $\norm{\qbar_k}_{H^1}$, the sum over $k\in \integers$ is clearly convergent and upon substitution into \eqref{e.HS}, the result \eqref{e:apriori2} follows. 
 \end{proof}
The following lemma is simply a restatement of the Kolmogorov continuity criterion \cite{revuz99}. 
\begin{lemma}
\label{l.kolmogorov}
Suppose $\{\phi(t)\}_{t\in [0,T]}$ is a complex valued stochastic process, such that for every $q>2$ there exists $K_q$ satisfying
\begin{align*}
\left( \E |\phi(t)|^q \right)^{1/q} &\leq K_q \left(\E|\phi(t)|^2 \right)^{1/2}\;,\\
\left( \E |\phi(t)-\phi(s)|^q \right)^{1/q} &\leq K_q \left(\E|\phi(t)-\phi(s)|^2 \right)^{1/2}\;,
\end{align*}
for any $s,t \in [0,T]$. Suppose furthermore that there exists $\delta>0$, $K_0>0$ such that
\[
\E|\phi(t)-\phi(s)|^2 \leq K_0 |t-s|^\delta\;,
\]
for any $s,t \in [0,T]$, where the constant $K_0$ depends only on the sequence $K_q$. Then for any $p > 0$ there exists $C>0$ such that
\[
\E \suptime |\phi(t)|^p \leq C(K_0+\E|\phi(0)|^2)^{p/2}\;.
\]  
\end{lemma}
The next and final result is needed in order to trade some regularity of a pair of functions for some extra decay on the Fourier modes of products of those functions. 
\begin{lemma}
\label{l.split}
Suppose $f,g \in H^1$ taking values in $\reals$, then for any $\nu \in [0,1]$ and each $k\in \integers$, we have that 
\begin{equation}
|\innerprod{fe_k,g}|^2 \lesssim (1\wedge |k|^{-2\nu}) \left(\norm{f}\norm{g}\right)^{2-2\nu} \left( \norm{f}_{H^1}\norm{g}_{H^1}\right)^{2\nu}   \;.
\end{equation}
\end{lemma}
\begin{proof}
We have that 
\begin{align*}
|\innerprod{fe_k,g}|^2 &= |\innerprod{fe_k,g}|^{2-2\nu}|\innerprod{fg,e_{-k}}|^{2\nu}\\
& = (1+k^2)^{-\nu}|\innerprod{fe_k,g}|^{2-2\nu}|\innerprod{(1-\del_x^2)^{1/2}(fg),e_{-k}}|^{2\nu}\\
&\lesssim (1\wedge|k|^{-2\nu}) \norm{f e_k}^{2-2\nu}\norm{g}^{2-2\nu} \norm{fg}_{H^1}^{2\nu}\\
&\lesssim  (1\wedge|k|^{-2\nu}) \norm{f}^{2-2\nu}\norm{g}^{2-2\nu}\norm{f}_{H^1}^{2\nu} \norm{g}_{H^1}^{2\nu}\;.
\end{align*}
In the last inequality we have used the fact that $H^1$ is a Banach algebra \cite{adams75}. This proves the lemma. 
\end{proof}
We now have all the necessary machinery to prove our first theorem.
\begin{proof}[Proof of Theorem \ref{thm.stronglimit}]
To start off, we take the object we wish to bound and split it into two parts. Using the identity
\[
\norm{\cdot}_{H^{-s}}^2 = \sum_{m \in \integers} |\innerprod{\cdot,e_m}|^2(1+m^2)^{-s}\;,
\]
We obtain
\begin{align*}
\E \suptime \norm{u_\eps(t) -u(t)}_{H^{-s}}^2
\lesssim &
\sum_{|m|<\eps^{-\beta}} \E \suptime |\innerprod{u_\eps(t)-u(t),e_m}|^2(1+m^2)^{-s}\\ 
+& \E \suptime \sum_{|m| \geq \eps^{-\beta}}  |\innerprod{u_\eps(t)-u(t),e_m}|^2(1+m^2)^{-s}
\end{align*}
for any $\beta \in (0,1)$. The idea of the proof is to use standard homogenisation techniques for the low modes ($|m| < \eps^{-\beta}$), while using rather
soft \textit{a priori} bounds for the high modes ($|m|\geq\eps^{-\beta}$). We then choose $\beta$ in the right way to balance the two contributions. We shall bound the low modes first. Here, we use the fact that
\begin{align*}
\innerprod{u_\eps(t),e_m} = \sum_k \int_0^t \innerprod{q_k^\eps e_k,S_\eps^*(t-s)e_m}\,dW_k(s) \;,
\end{align*}
and then approximate the semigroup as follows
\begin{align*}
S_\eps^*(t-s)e_m = &\rho(x/\eps)e_m(x) e^{-\mu m^2(t-s)}+\rhobar_{(t-s)/\eps^2}(x/\eps)e_m(x) +R_\eps(x,t-s)\;,
\end{align*}
so that 
\begin{align*}
\innerprod{u_\eps(t),e_m} &=  \sum_k \innerprod{q_k^\eps e_k ,\rho^\eps e_m} \int_0^t e^{-\mu m^2 (t-s )} dW_k(s) \\
&\quad+\sum_k \int_0^t \innerprod{q_k^\eps e_k,\rhobar_{(t-s)/\eps^2}^\eps e_m}dW_k(s)\\
&\quad+\sum_k \int_0^t \innerprod{q_k^\eps e_k,R_\eps(x,t-s)}dW_k(s)\;,
\end{align*}
where $\rho^\eps(x) = \rho(x/\eps)$ and similarly for all other instances of the superscript $\eps$. We can simplify the terms above using the fact that, for fixed $|m| < \eps^{-\beta} \ll \eps^{-1}$ and varying $k\in \integers$ the expression $\innerprod{q_k^\eps e_k, \rho^\eps e_m}$ is zero, unless $k = m + l/\eps$ for some $l \in \integers$. We can see this, for example, by performing a Fourier expansion on both $q_k$ and $\rho$. Moreover, 
\[
\sum_{k \in \integers} \innerprod{q_k^\eps e_k, \rho^\eps e_m} F_k = \sum_{l \in \integers} \innerprod{q_{m+l/\eps} e_l,\rho} F_{m+l/\eps}\;, 
\]
for any sequence $\{F_k\}_{k\in\integers}$. Therefore, 
\begin{align*}
 &\sum_k \innerprod{q_k^\eps e_k,\rho^\eps e_m} \int_0^t e^{-\mu m^2 (t-s )} dW_k(s)\\
& = \innerprod{q_m,\rho}\int_0^t e^{-\mu m^2 (t-s)} dW_m(s) + \sum_{l \neq 0}\innerprod{q_{m+l/\eps}e_l,\rho}\int_0^t e^{-\mu m^2 (t-s)} dW_{m+l/\eps}(s)\;.
\end{align*}
Similarly, we can write
\begin{align*}
 \sum_k \int_0^t \innerprod{q_k^\eps e_k&,\rhobar_{(t-s)/{\eps^2}}^\eps e_m} dW_k(s)\\
& =   \sum_{l}\int_0^t \innerprod{q_{m+l/\eps}e_l,\rhobar_{(t-s)/{\eps^2}}} dW_{m+l/\eps}(s)\;.
\end{align*}
It is easy to see that $\innerprod{u(t),e_m} = \innerprod{q_m,\rho}\int_0^t e^{-\mu m^2 (t-s)} dW_m(s) $ and we can therefore write
\begin{equs}
\innerprod{u_\eps(t)-u(t),e_m} &=  \sum_{l \neq 0}\innerprod{q_{m+l/\eps}e_l,\rho}\int_0^t e^{-\mu m^2 (t-s)} dW_{m+l/\eps}(s)\\
&\quad+\sum_{l} \int_0^t \innerprod{q_{m+l/\eps}e_l,\rhobar_{(t-s)/\eps^2}}dW_{m+l/\eps}(s)\\
&\quad+\sum_k \int_0^t \innerprod{q_k^\eps e_k,R_\eps(t-s)}dW_k(s)\;.
\end{equs}
We then bound separately each of the three sums in this expression. 
In order to streamline the presentation, we state these bounds as separate lemmas, the proof of which is given below.

\begin{lemma}\label{lem:claim1}
For $\eps|m| < 1/2$, one has the bound
\minilab{e:claims}
\begin{equ}
\E \suptime \left| \sum_{l \neq 0}\innerprod{q_{m+l/\eps}e_l,\rho}\int_0^t e^{-\mu m^2 (t-s)} dW_{m+l/\eps}(s) \right|^2 \lesssim \frac{\eps^{2\alpha}}{1\vee m^2}\;, \label{e.claim1}
\end{equ}
for any $\alpha > 0$.
\end{lemma}

\begin{lemma}\label{lem:claim2}
For $\eps|m| < 1/2$, one has the bound
\minilab{e:claims}
\begin{equ}
\E \suptime \left| \sum_{l\in\integers} \int_0^t \innerprod{q_{m+l/\eps}e_l,\rhobar_{(t-s)/\eps^2}}dW_{m+l/\eps}(s)\right|^2 \lesssim \eps^{2-2\delta}\;, \label{e.claim2} \end{equ}
 for any sufficiently small $\delta>0$.
\end{lemma}

\begin{lemma}\label{lem:claim3}
For $\eps|m| < 1/2$,  the bound
\minilab{e:claims}
\begin{equ}
\E \suptime \left|\sum_k \int_0^t \innerprod{q_k^\eps e_k,R_\eps(t-s)}dW_k(s)\right|^2 \lesssim {\eps^{4\alpha} + \eps^2 \over \eps^{7\delta}}(1\vee m^{2+\delta})\;,\label{e.claim3}
\end{equ}
holds for any sufficiently small $\delta>0$ and for any $\alpha > 0$
\end{lemma}

We now use these bounds to prove the claim made in the statement of the theorem in the case $\alpha \in (0,1/2]$, and the case $\alpha \in (1/2,1)$ will follow similarly. Inserting the bounds above into  
\begin{equs}
\sum_{|m|<\eps^{-\beta}} &\E \suptime |\innerprod{u_\eps(t)-u(t),e_m}|^2(1+m^2)^{-s}\\ 
& \lesssim \eps^{2\alpha} \sum_{|m|<\eps^{-\beta}} \frac{(1+m^2)^{-s}}{1\vee m^2} + \eps^{2-2\delta} \sum_{|m|<\eps^{-\beta}}(1+m^2)^{-s}\\ &\quad+  \eps^{4\alpha-7\delta} \sum_{|m|<\eps^{-\beta}} (1\vee m^{2+\delta})(1+m^2)^{-s} \\
&\lesssim  \eps^{2\alpha} + \eps^{2-\beta -2\delta} + \eps^{4\alpha-(3-2s)\beta - (2\beta+7)\delta}\;,\label{e:finalbound}
\end{equs}
for any $s>0$.  
For the high modes on the other hand, we have the straightforward bound 
\begin{equs}
\E \suptime \sum_{|m|\geq\eps^{-\beta}}&  |\innerprod{u_\eps(t)-u(t),e_m}|^2(1+m^2)^{-s}\label{e:finalbound2}\\ 
&\lesssim \eps^{2\beta s} \left(\E \suptime \norm{u_\eps(t)}^2+\E \suptime \norm{u(t)}^2 \right) \lesssim \eps^{2\beta s+4\alpha-2-\delta'}\;,
\end{equs}
where we have used Lemma \ref{l.apriori} combined with the fact that 
\[
\E \suptime \norm{u(t)}^2 \lesssim1\;,
\] 
which is easily verified. Since $\delta$ and $\delta'$ can be chosen arbitrarily small and since $\beta \in (0,1)$, both the low modes and high modes will be bounded by a multiple of $\eps^\theta$, where $\theta < \theta_0$ and 
\[
\theta_0 = \min \left\{ 2\alpha , 1+2\alpha-\beta , 4\alpha-(3-2s)\beta, 2\beta s + 4\alpha-2 \right\}\;.
\]
Since $\alpha > 0$ and $\beta \in (0,1)$ we will find $\theta_0 > 0$ provided that $4\alpha-(3-2s)\beta > 0$ and $2\beta s + 4\alpha -2>0$ are both satisfied. That is, the result \eqref{e.stronglimit} will hold for $s >0$ if we can find $\beta \in (0,1)$ such that 
\begin{equation}
\frac{1-2\alpha}{s} <  \beta < \frac{4\alpha}{3-2s}  \;.
\end{equation}
A simple diagram verifies that, for fixed $\alpha \in (0,1/2]$ we can always find such a $\beta$ provided $s > s_\alpha$ where
\[
s_\alpha = 0 \vee \frac{3}{2}(1-2\alpha)\;,
\]
as in the statement of the theorem.
Moreover, one can also show that the optimal value of $\theta$ is given by
\[
\theta_0(s,\alpha) = 2\alpha \wedge \left(4\alpha - 2 + \frac{4s}{3} \right) = 2\alpha\wedge \left( \frac{4}{3}(s-s_\alpha) \right)\;.
\]
which only takes positive values when $s> s_\alpha$.

The case $\alpha \in (1/2,1)$ is actually slightly easier, and we obtain the same bounds on the low and high modes as in \eref{e:finalbound}
and \eref{e:finalbound2}, but with $\alpha$ replaced by $1/2$ and $\delta'=0$. Hence, the result \eqref{e.stronglimit} will hold for $s>0$ if we can find $\beta \in (0,1)$ such that 
\begin{equation}
0<  \beta < \frac{2}{3-2s}\;.
\end{equation}
One can always find such a $\beta$, provided $s>0$ is small enough. 
Moreover, one can also show that the optimal value of $\theta$ is given  in this case by
\[
\theta_0(s)  = 1 \wedge \frac{4}{3}s\;.
\]
This proves the claims made in the statement of the theorem.  
\end{proof}

It thus remain to show that the bounds \eref{e:claims} hold. 

\begin{proof}[Proof of Lemma~\ref{lem:claim1}]
Starting with \eqref{e.claim1}, we have that 
\begin{equs}
\E \suptime &\left| \sum_{l \neq 0}\innerprod{q_{m+l/\eps}e_l,\rho}\int_0^t e^{-\mu m^2 (t-s)} dW_{m+l/\eps}(s) \right|^2\\ &= \left(\sum_{l\neq0}|\innerprod{q_{m+l/\eps}e_l,\rho}|^2\right) \E \suptime \left|\int_0^t e^{-\mu m^2 (t-s)} dB(s)\right|^2\\
&\quad \lesssim \sum_{l\neq0}\frac{|\innerprod{q_{m+l/\eps}e_l,\rho}|^2}{1\vee m^2}\;.
\end{equs}
If $\alpha \in (1/2,1)$ then
\begin{align*}
\sum_{l\neq0}|\innerprod{q_{m+l/\eps}e_l,\rho}|^2 &\lesssim \sum_{l\neq0} \norm{q_{m+l/\eps}}^2 \norm{\rho}^2_\infty \lesssim \sum_{l\neq0}1 \wedge |m+l/\eps|^{-2\alpha}\;.
\end{align*}
Assume for now that $m \geq 0$, the case $m<0$ will follow similarly. Recalling that $\eps|m| < 1/2$ by assumption, we can bound the above by 
\[
\eps^{2\alpha} \sum_{l\neq 0 }|\eps m + l|^{-2\alpha} \lesssim \eps^{2\alpha} \left( \sum_{l \geq 1}  |l|^{-2\alpha} + \sum_{l \geq 1} |l-1/2|^{-2\alpha} \right) \lesssim \eps^{2\alpha}\;.
\]
Now suppose $\alpha \in (0,1/2]$. Using Lemma \ref{l.split} with $\nu=1$, we have the following bound
\begin{align*}
\sum_{l\neq0}\innerprod{q_{m+l/\eps}e_l,\rho}^2 &\lesssim \sum_{l\neq0}(1\wedge|l|^{-2}) \norm{q_{m+l/\eps}}^2 \norm{\qbar_{m+l/\eps}}_{H^1}^{2}\norm{\rho}_{H^1}^{2}\\
&\lesssim \sum_{l\neq0}|l|^{-2}\norm{q_{m+l/\eps}}^2\;.
\end{align*}
The boundedness of $\norm{\rho}_{H^1}$ is guaranteed by Assumption \ref{ass.elliptic} and the uniform boundedness of $\norm{\qbar_k}_{H^1}$ is guaranteed by Assumption \ref{ass.strongnoise}. Moreover, we have that 
\[
\sum_{l\neq0}|l|^{-2} \norm{q_{m+l/\eps}}^2 \lesssim \sum_{l\neq0} |l|^{-2} |m+l/\eps|^{-2\alpha}\;.
\]
We will now show that this sum decays like $\eps^{2\alpha}$. Since $\eps |m| < 1/2$ it follows that $|\eps m + l |^{-2\alpha} \leq |l-1/2|^{-2\alpha}$ for $|l| \geq 1$. Therefore
\begin{align*}
&\sum_{l\neq0} |l|^{-2} |m+l/\eps|^{-2\alpha} = \eps^{2\alpha} \sum_{l\neq0}  |l|^{-2}|\eps m+l|^{-2\alpha} \lesssim \eps^{2\alpha} \sum_{l\neq0}  |l|^{-2}\;.
\end{align*}
This proves \eqref{e.claim1}. 
\end{proof}

\begin{proof}[Proof of Lemma~\ref{lem:claim2}]
For both \eqref{e.claim2} and \eqref{e.claim3} we are trying to bound objects of the form
\[
\phi(t) = \sum_k \int_0^t f_k(t-r)dw_k(r)\;,
\]
where the $w_k$ are independent Brownian motions and each $f_k$ takes values in $\complex$. Since $\phi(t)$ is a Gaussian process, we may apply Lemma \ref{l.kolmogorov}. Thus, if we can show that 
\[
\E |\phi(t)-\phi(s)|^2 \leq K_\delta(\eps) |t-s|^\delta\;,
\]
then it follows that
\[
\E \suptime |\phi(t)|^2 \lesssim K_\delta(\eps)\;. 
\]
In general, we have that
\begin{align*}
\E |\phi(t)-\phi(s)|^2 &= \E \left| \sum_{k} \int_s^t f_k(t-r) dw_k(r) + \int_0^s (f_k(t-r) - f_k(s-r))dw_k(r)   \right|^2 \\
&\quad\lesssim \sum_k \int_s^t |f_k(t-r)|^2 dr + \sum_k \int_0^s |f_k(t-r)-f_k(s-r)|^2 dr\;.
\end{align*}
Note that  the Brownian motions $w_k$ are not truly independent due to the requirement $W_k = W_{-k}^*$. However, one can easily check that the above bound still holds. We then have that 
\begin{equation}
\label{e.claim2proof1}
\sum_k \int_s^t |f_k(t-r)|^2 dr = \sum_l \int_s^t |\innerprod{q_{m+l/\eps}e_l,\rhobar_{(t-r)/\eps^2}}|^2 dr\;.
\end{equation}
If $\alpha \in (1/2,1)$ then we can bound the above by
\begin{equation}\label{e.claim2proof1a}
\sum_{l} \norm{q_{m+l/\eps}}^2 \int_s^t \norm{\rhobar_{(t-r)/\eps^2}}^2 dr \;.
\end{equation}
From Lemma \ref{l.semigroupstar} we have that 
\begin{equation}\label{e.rhobar}
\norm{\rhobar_{r/\eps^2}}= \norm{S^*(r/\eps^2)(1-\rho)} \lesssim  \exp(-r/\eps^2)\;.  
\end{equation}
Moreover, since the sum over $l$ is finite when $\alpha \in (1/2,1)$ we can apply H\"older's inequality to \eqref{e.claim2proof1a} to obtain
\begin{equs}
\sum_{l} \norm{q_{m+l/\eps}}^2 \int_s^t \norm{\rhobar_{(t-r)/\eps^2}}^2 dr &\lesssim |t-s|^\delta \left(\int_s^t (\exp(-r/\eps^2))^{2/(1-\delta)} dr\right)^{1-\delta} \\
&\lesssim \eps^{2-2\delta}|t-s|^\delta \;. \label{e:bizarre}
\end{equs}
Now suppose $\alpha \in (0,1/2]$. Using Lemma \ref{l.split} with $\nu=1$ we can bound \eqref{e.claim2proof1} by 
\begin{equation}\label{e.claim2proof1b}
\sum_l  (1\wedge|l|^{-2}) \norm{q_{m+l/\eps}}^{2}_{H^1} \int_s^t  \norm{\rhobar_{(t-r)/\eps^2}}^{2}_{H^1}  dr 
\end{equation}
Since $\norm{\qbar_k}_{H^1}$ is bounded uniformly in $k$, the sum over $l$ is finite. Furthermore, from Lemma \ref{l.semigroupstar} we see that 
\begin{equation}\label{e.rhobarh1}
\norm{\rhobar_{r/\eps^2}}_{H^1} = \Bigl\|(1-\del_x^2)^{1/2}S^*(1) S^*(r/\eps^2 -1)(1-\rho)\Bigr\| \lesssim \exp(-r/\eps^2)\;.   
\end{equation}
Therefore, with an application of H\"older's inequality, we can bound \eqref{e.claim2proof1b} by
\begin{equ}
|t-s|^\delta\left( \int_s^t \norm{\rhobar_{(t-r)/\eps^2}}^{2/(1-\delta)}_{H^1} dr     \right)^{1-\delta} \lesssim |t-s|^{\delta} \eps^{2-2\delta}\;.
\end{equ}
We also have that 
\begin{equation}\label{e.claim2proof2}
\sum_k \int_0^s |f_k(t-r)-f_k(s-r)|^2 dr = \sum_l \int_0^s |\innerprod{q_{m+l/\eps}e_l,\rhobar_{(t-r)/\eps^2}-\rhobar_{(s-r)/\eps^2}}|^2 dr\;.
\end{equation}
If $\alpha \in (1/2,1)$ then, as in the estimation of \eqref{e.claim2proof1} we can bound the above by
\begin{equs}
\int_0^s \norm{\rhobar_{(t-r)/\eps^2}-\rhobar_{(s-r)/\eps^2}}^{2} dr &\lesssim \frac{|t-s|^\delta}{\eps^{2\delta}} \left(\suptime\norm{\del_t \rhobar_t} \right)^\delta \\
&\qquad \times\int_0^s \norm{\rhobar_{(t-r)/\eps^2}-\rhobar_{(s-r)/\eps^2}}^{2-\delta} dr \\
&\lesssim |t-s|^{\delta} \eps^{-2\delta} \int_0^T \norm{\rhobar_{r/\eps^2}}^{2-\delta} dr \lesssim \eps^{2-2\delta}|t-s|^\delta\;.
\end{equs}
Here we have used the fact that 
\begin{equ}
\norm{\rhobar_{(t-r)/\eps^2}-\rhobar_{(s-r)/\eps^2}}^\delta \leq |t-s|^\delta \suptime \norm{\del_t \rhobar_{t/\eps^2}}^\delta\lesssim  \frac{|t-s|^\delta}{\eps^{2\delta}} \suptime \norm{\del_t \rhobar_t}^\delta \;,
\end{equ}
and that $\norm{\del_t \rhobar_t(x)}$ is bounded uniformly in time, which follows from the smoothness of $b$ and $\sigma$. 
Now suppose $\alpha \in (0,1/2]$. Using Lemma \ref{l.split} with $\nu=3/4$ and arguments similar to those used in the estimation of \eqref{e.claim2proof1} we can bound \eqref{e.claim2proof2} by
\begin{align*}
\int_0^s &\norm{\rhobar_{(t-r)/\eps^2}-\rhobar_{(s-r)/\eps^2}}^{1/2}  \norm{\rhobar_{(t-r)/\eps^2}-\rhobar_{(s-r)/\eps^2}}^{3/2}_{H^1} \,dr\\
&\lesssim \frac{|t-s|^\delta}{\eps^{2\delta}} \left(\suptime\norm{\del_t \rhobar_t} \right)^\delta \\
&\qquad \times\int_0^s \norm{\rhobar_{(t-r)/\eps^2}-\rhobar_{(s-r)/\eps^2}}^{1/2-\delta}  \norm{\rhobar_{(t-r)/\eps^2}-\rhobar_{(s-r)/\eps^2}}^{3/2}_{H^1} \,dr\\
& \lesssim \eps^{2-2\delta}|t-s|^\delta\;.
\end{align*}
To bound the integral term, we have used estimates \eqref{e.rhobar} and $\eqref{e.rhobarh1}$. Putting this all together, we have that $K_\delta(\eps) = \eps^{2-2\delta}$, which proves estimate \eqref{e.claim2}.  
\end{proof}

\begin{proof}[Proof of Lemma~\ref{lem:claim3}]
We use the same strategy as in the proof of Lemma~\ref{lem:claim2}. We see that,
\begin{equation}\label{e.claim3proof1}
\sum_k \int_s^t |f_k(t-r)|^2 dr = \sum_k \int_s^t |\innerprod{q_{k}^\eps e_k,R_\eps(t-r)}|^2 dr\;.
\end{equation}
If $\alpha \in (1/2,1)$ then we can bound the above by
\[
\left(\sum_k \norm{q_k}^2\right) 
\int_s^t \norm{R_\eps(t-r)}^2 dr \lesssim \eps^2 m^2 |t-s|\;.
\]
Here we have used the finiteness of the sum over $k$ as well as Lemma \ref{l.remainder} to bound the remainder term uniformly in time. Suppose that $\alpha \in (0,1/2]$. Using the Lemma \ref{l.split}, we can bound \eqref{e.claim3proof1} by
\begin{equs}
\sum_k &(1\wedge|k|^{-2\nu})\norm{q_k}^2 \norm{\qbar_k^\eps}_{H^1}^{2\nu} \int_s^t \norm{R_\eps(t-r)}^{2-2\nu}\norm{R_\eps(t-r)}_{H^1}^{2\nu} dr\\
&\lesssim \sum_k (1\wedge|k|^{-2\nu}){\norm{q_k}^2 \over \eps^{2\nu}} |t-s|^{\delta} \left(  \int_s^t \norm{R_\eps(t-r)}^{2-2\nu \over 1-\delta}\norm{R_\eps(t-r)}_{H^1}^{2\nu\over 1-\delta} dr   \right)^{1-\delta}\;,
\end{equs}
for any $\nu \in [0,1]$. Here we have used the fact that $\norm{\qbar_k^\eps}_{H^1} \leq \eps^{-1} \norm{\qbar_k}_{H^1} \lesssim \eps^{-1}$ and then applied H\"older's inequality to the integral. Choose $\nu \in (0,1/2)$ such that $\alpha+\nu > 1/2$, to guarantee that the above sum is bounded. Using the estimates on the remainder $R_\eps$ given in Lemma \ref{l.remainder} we have that 
\begin{align*}
\int_s^t \norm{R_\eps(t-r)}^{(2-2\nu)/(1-\delta)}&\norm{R_\eps(t-r)}_{H^1}^{2\nu/(1-\delta)} dr\\ &\lesssim  (\eps m)^{(2-2\nu)/(1-\delta)} \int_0^T \norm{R_\eps(r)}^{2\nu/(1-\delta)} dr\;.
\end{align*}
For any $\nu \in [0,1/2)$, we can choose $\delta$ small enough that $2\nu/(1-\delta)<1$ and hence, by Jensen's inequality
\[
\int_0^T \norm{R_\eps(r)}^{2\nu/(1-\delta)}dr \leq \left(\int_0^T \norm{R_\eps(r)}dr\right)^{2\nu/(1-\delta)} \lesssim m^{2\nu/(1-\delta)}\;,
\]
which follows from Lemma \ref{l.remainder}. Therefore, we can bound \eqref{e.claim3proof1} by 
\[
\eps^{-2\nu}  |t-s|^\delta \eps^{2-2\nu} m^2\lesssim \eps^{2-4\nu}m^2 |t-s|^\delta\;.
\] 
We then substitute $\nu = 1/2-\alpha+\delta$ and ensure $\delta$ is small enough so that all the above conditions on $\nu$ are satisfied. 
\\
We also have that 
\begin{equation}\label{e.claim3proof2}
\sum_k \int_0^s |f_k(t-r)-f_k(s-r)|^2 dr = \sum_k \int_0^s |\innerprod{q_{k}^\eps e_k,R_\eps(t-r)-R_\eps(s-r)}|^2 dr\;.
\end{equation}
If $\alpha \in (1/2,1)$ then, as in the previous step we can bound the above by a multiple of
\begin{equs}
\int_0^s &\norm{R_\eps(t-r)-R_\eps(s-r)}^2  dr \\
&\qquad\lesssim |t-s|^\delta \suptime \norm{\del_t R_\eps(t)}^\delta \int_0^s \norm{R_\eps(t-r)-R_\eps(s-r)}^{2-\delta}dr\;.
\end{equs}
Using the estimates on $R_\eps$ given in Lemma \ref{l.remainder}, we can bound this by a constant multiple of
\[
\eps^{2-3\delta}m^{2+\delta}|t-s|^\delta\;.
\]
If $\alpha \in (0,1/2]$ on the other hand, we can bound \eqref{e.claim3proof2} by
\begin{equs}
&\sum_k (1 \wedge |k|^{-2\nu})\norm{q_k}^2 \eps^{-2\nu} \\
&\qquad\times\int_0^s \norm{R_\eps(t-r)-R_\eps(s-r)}^{2-2\nu}\norm{R_\eps(t-r)-R_\eps(s-r)}_{H^1}^{2\nu} \,dr \;.
\end{equs}
As before, we choose $\nu \in (0,1/2)$ such that $\alpha+\nu > 1/2$, this guarantees the above sum is bounded. Moreover, we can bound the above integral by
\[
|t-s|^\delta \suptime \norm{\del_t R_\eps(t)}^\delta \int_0^s \norm{R_\eps(t-r)-R_\eps(s-r)}^{2-2\nu-\delta} \norm{R_\eps(t-r)-R_\eps(s-r)}_{H^1}^{2\nu} dr\;.
\]
Using the estimates on $R_\eps$ given in Lemma \ref{l.remainder}, we can bound this by a constant multiple of
\[
\eps^{2-2\nu-3\delta} m^{2-2\nu+\delta}|t-s|^\delta \int_0^T \norm{R_\eps(r)}^{2\nu} dr\;.
\]
And, by Jensen's inequality, since $2\nu<1$, we can bound the above by
\[
\eps^{2-2\nu-3\delta} m^{2-2\nu+\delta}|t-s|^\delta \left( \int_0^T \norm{R_\eps(r)} dr\right)^{2\nu} \lesssim \eps^{2-2\nu-3\delta} m^{2+\delta}|t-s|^\delta\;.
\]
We then substitute $\nu = 1/2-\alpha+\delta$ and ensure $\delta$ is small enough so that the above condition on $\nu$ are satisfied. Hence, we have that
\[
K_\delta(\eps) = \eps^{2-4\nu -3\delta}m^{2+\delta} = \eps^{4\alpha-7\delta} m^{2+\delta}\;, 
\] 
which proves estimate \eqref{e.claim3}. 
\end{proof}
We now concentrate on the second convergence theorem, where we assume that the noise satisfies $\innerprod{q_k,\rho}=0$ for all $k\in \integers$. Before proving the theorem, we give a formal argument to describe how the proof works. It is clear from the proof of the previous theorem that we can formally write
\[
\innerprod{u_\eps(t),e_m} = \sum_{l\neq0} \innerprod{q_{m+l/\eps} e_l , \rho } \int_0^t e^{-\mu m^2 (t-s)} dW_{m+l/\eps}(s) + \bigoh(\eps^\theta)
\]
for some $\theta>0$, provided $m$ is not too large. The previous theorem tells us that the first term above will decay with $\eps$ to zero. However, with Assumption \ref{ass.weaknoise} in place, we have precise control over how this term tends to zero. In fact, we have that
\begin{equs}
 \innerprod{\eps^{-\alpha} u_\eps(t),e_m} &= \sum_{l\neq0} \eps^{-\alpha} \innerprod{q_{m+l/\eps} e_l , \rho } \int_0^t e^{-\mu m^2 (t-s)} dW_{m+l/\eps}(s) + \bigoh(\eps^{\theta-\alpha})\\
&=  \sum_{l\neq0} \eps^{-\alpha} (m+l/\eps)^{-\alpha} \innerprod{(m+l/\eps)^{\alpha} q_{m+l/\eps} e_l , \rho } \\
&\quad \times\int_0^t e^{-\mu m^2 (t-s)} dW_{m+l/\eps}(s) + \bigoh(\eps^{\theta-\alpha})\;,
\end{equs}
and all the terms in the sum are no longer decaying with $\eps$. Now, since a convergent sum of complex OU processes is a complex OU process, we can find a Brownian motion $\tildeW_m$ such that the above is equal in distribution to
\[
\Lambda_{\eps,m} \int_0^t e^{-\mu m^2 (t-s)} d\tildeW_{m}(s) + \bigoh(\eps^{\theta-\alpha})
\]
where we denote
\[
\Lambda_{\eps,m} = \left( \sum_{l\neq0} \eps^{-2\alpha} |m+l/\eps|^{-2\alpha} |\innerprod{(m+l/\eps)^{\alpha} q_{m+l/\eps} e_l , \rho }|^2 \right)^{1/2}\;.
\]
If we can justify taking the limit inside the above sum then it is clear that 
\[
\lim_{\eps \to 0} \Lambda_{\eps,m} =  \left( \sum_{l\neq0} |l|^{-2\alpha} |\innerprod{ \qbar \rho , e_{-l} }|^2 \right)^{1/2} = \norm{\qbar \rho}_{-\alpha}\;,
\]
recalling that $|k|^\alpha q_k \to \qbar$ in $\Ltwopi$. If we can also adjust our estimates on the remainder to ensure that $\theta > \alpha$, so that $\eps^{\theta-\alpha}$ does indeed decay, then formally we have shown that $\innerprod{u_\eps(t),e_m}$ is equal in distribution to a process that converges to
\[
\norm{\qbar \rho}_{-\alpha} \int_0^t e^{-\mu m^2 (t-s)} d\tildeW_{m}(s)\;,
\] 
which is the $m$-th Fourier mode of the solution to the limiting SPDE \eqref{e.weaklimitSPDE}. Of course, there are several caveats with this argument. Most importantly, the Brownian motions $\tildeW_m$ are defined in such a way that their distribution changes as $\eps$ tends to zero and consequently, the limit above does not make sense. The correct way to proceed is actually \emph{backwards}. That is, we fix a sequence of Brownian motions $\tildeW_m$ that are used to construct the limiting SPDE \eqref{e.weaklimitSPDE}. We then construct a sequence of processes $\tildeu_\eps$ equal in law to $u_\eps$ defined in such a way that when we perform the above calculations, the resulting OU process (driven by $\tildeW_m$) does not depend on $\eps$. This is made rigorous below.     

\begin{rmk}
It is clear from the preceding argument that no stronger type of convergence is possible in the context of Theorem \ref{thm:weaklimit}. In particular, we see that the limiting term in $\innerprod{\eps^{-\alpha} u_\eps,e_m}$ is an OU process determined by $\{W_{m+l/\eps} \}$ for each $l\in \integers$. Hence, even when $\eps$ is near zero, the contributing BMs are always changing; we will never be able to pin down the limiting process to a fixed location of our probability space so convergence in probability is not possible.    
\end{rmk}

\begin{proof}[Proof of Theorem \ref{thm:weaklimit}]
The process $\tildeu_\eps$ will be defined using two sequences of BMs, namely $\{\tildeW_m\}_{m \in \integers}$ and $\{ B_k^\eps\}_{k \in \integers}$, that live on a different probability space. Given a sequence $\{\tildeW_m\}_{m \in \integers}$ of i.i.d.\ complex-valued Wiener processes (modulo the reality condition $\tildeW_m = \tildeW_{-m}^\star$, we construct a sequence $\{ B_k^\eps \}_{k \in \integers}$ of i.i.d.\ complex-valued Wiener processes (again modulo the corresponding reality condition) such that $(\tildeW, B^\eps)$ are jointly Gaussian with the covariance structure given by
\[
\E \tildeW_m(t) B_k^\eps(s) =\begin{cases}
\frac{\lambda^l_{\eps,m}}{\Lambda_{\eps,m}} (t \wedge s) & \text{if $k=m+l/\eps$ for some $l \in \integers$},\\
0 & \text{otherwise},
\end{cases}
\]
where $\lambda_{\eps,m}^l = \eps^{-\alpha}\innerprod{q_{m+l/\eps} e_l,\rho}$. Such a construction is possible due to the fact that $\Lambda_{\eps,m}^2 = \sum_{l} |\lambda^l_{\eps,m}|^2$ by definition. In the new probability space, one should view the sequence $\{B_k^\eps\}$ as playing the role of the sequence $\{W_k \}$ in the old space. We can now define $\tildeu_\eps$ by its Fourier coefficients. For $|m|<\eps^{-\beta}$ set
\begin{equs}
\innerprod{\tildeu_\eps(t),e_m} &=  \eps^\alpha \Lambda_{\eps,m} \int_0^t e^{-\mu m^2 (t-s)}d\tildeW_m(s) + 
  \sum_{k} \int_0^t \innerprod{q_k^\eps e_k,R_\eps(t-s)} dB_k^{\eps}(s)\\ 
 &\quad + 
   \sum_{k} \int_0^t \innerprod{q_k^\eps e_k,\rhobar_{(t-s)/\eps^2}^\eps } dB_k^{\eps}(s) 
\;.
\end{equs}
For $|m|\geq \eps^{-\beta}$ on the other hand, we simply set 
\[
\innerprod{\tildeu_\eps(t),e_m} = \innerprod{w_\eps(t),e_m}\;,
\]
where $w_\eps$ solves the SPDE \eqref{e:formulationSPDE} with $\{W_k\}$ replaced by $\{B_k^\eps \}$. One can verify that $u_\eps \stackrel{\tiny\text{law}}{=} \tildeu_\eps$ by checking that
\[
\E \innerprod{u_\eps(t),e_m}\innerprod{u_\eps(s),e_n} = \E \innerprod{\tildeu_\eps(t),e_m}\innerprod{\tildeu_\eps(s),e_n}
\] 
for all choices of $t,s \in [0,T]$ and $n,m \in \integers$. We define $v(t)$ as the mild solution to SPDE \eqref{e.weaklimitSPDE}. In particular, we have that 
\[
\innerprod{v(t),e_m} = \norm{\qbar \rho}_{-\alpha} \int_0^t e^{-\mu m^2 (t-s)} d\tildeW_m(s) \;,
\]
for each $m\in \integers$.

We shall now prove that $\eps^{-\alpha} \tildeu_\eps \to v$ in the required sense. Firstly, we split the problem into high and low modes
\begin{align*}
 \E\sup_{t \in [0,T]} &\norm{ \eps^{-\alpha}\tildeu_\eps(t) - v(t)   }_{H^{-s}}^2 \\ 
&\lesssim 
 \sum_{|m|<\eps^{-\beta}} \E\sup_{t \in [0,T]}  |\innerprod{\eps^{-\alpha}\tildeu_\eps(t) -v(t),e_m}|^2 (1+m^2)^{-s}\\
  &\quad + 
  \E\suptime\sum_{|m| \geq \eps^{-\beta}}   |\innerprod{\eps^{-\alpha}\tildeu_\eps(t) -v(t),e_m}|^2 (1+m^2)^{-s}\;.
\end{align*}
We can bound the low modes in the following way
\begin{align*}
\E\sup_{t \in [0,T]}  &|\innerprod{\eps^{-\alpha}\tildeu_\eps(t) -v(t),e_m}|^2\\ 
&\lesssim
\E\sup_{t \in [0,T]}  \left|  (\Lambda_{\eps,m}-\norm{\qbar \rho}_{-\alpha})\int_0^t e^{\mu m^2 (t-s)} d\tildeW_m(s)\right|^2\\ 
&\quad + 
 \E\sup_{t \in [0,T]} \left| \eps^{-\alpha} \sum_{k}  \int_0^t \innerprod{q_k^\eps e_k,R_\eps(t-s)} dB_k^{\eps}(s)   \right|^2\\
 &\quad + 
 \E\sup_{t \in [0,T]} \left| \eps^{-\alpha} \sum_{l}  \int_0^t \innerprod{q_{m+l/\eps}e_l,\rhobar_{(t-s)/\eps^2}} dB_k^{\eps}(s)   \right|^2\;.
 \end{align*}
 However, it is clear that 
 \[
 \E\sup_{t \in [0,T]}  \left|  (\Lambda_{\eps,m}-\norm{\qbar \rho}_{-\alpha})\int_0^t e^{-\mu m^2 (t-s)} d\tildeW_m(s)\right|^2 \lesssim   |\Lambda_{\eps,m}-\norm{\qbar \rho}_{-\alpha}|^2 ({1\wedge m^{-2}})\;.
 \]
And from Theorem \ref{thm.stronglimit} we have that the two estimates
\begin{align*}
 &\E\sup_{t \in [0,T]} \left|  \sum_{k}  \int_0^t \innerprod{q_k^\eps e_k,R_\eps(t-s)} dB_k^{\eps}(s)   \right|^2 \lesssim (\eps^{4\alpha}\vee \eps^2)\eps^{-7\delta} m^{2+\delta} \\
 &\E\sup_{t \in [0,T]} \left|   \sum_{l}  \int_0^t \innerprod{q_{m+l/\eps}e_l,\rhobar_{(t-s)/\eps^2}} dB_{m+l/\eps}^{\eps}(s)   \right|^2  \lesssim \eps^{2-\delta}
\end{align*}
hold for sufficiently small $\delta>0$. Using these estimates, when $ |m| < \eps^{-\beta}$, we have that 
\begin{equs}
 \sum_{|m|<\eps^{-\beta}} &\E\sup_{t \in [0,T]}  |\innerprod{\eps^{-\alpha}\tildeu_\eps(t) -v(t),e_m}|^2 (1+m^2)^{-s} \\
 &\lesssim 
  \sum_{|m|<\eps^{-\beta}}|\Lambda_{\eps,m}-\norm{\qbar \rho}_{-\alpha}|^2({1\wedge m^{-(2+2s)}}) \label{e.weaklowmodes} \\ &\quad + \eps^{2-2\alpha-\delta} \sum_{|m|<\eps^{-\beta}} (1\wedge m^{-2s})
  + (\eps^{2\alpha}\vee \eps^{2-2\alpha})\eps^{-7\delta} \sum_{|m|<\eps^{-\beta}} (1\wedge m^{2-2s+\delta}) \;.
\end{equs}
Firstly, we would like to show that the first sum in the expression above tends to zero as $\eps\to0$, by taking the limit inside the sum over $m$. Now, since $\norm{q_k} \lesssim 1\wedge |k|^{-\alpha}$ for each $k\in\integers$, we have that
\begin{align*}
\Lambda_{\eps,m}^2 &= \eps^{-2\alpha} \sum_{l\neq0} |\innerprod{q_{m+l/\eps} \rho, e_{-l}}|^2
 \lesssim  \sum_{l\neq0} |\eps m + l|^{-2\alpha} \innerprod{\qbar_{m+l/\eps} \rho, e_{-l}}^2\\
&\lesssim  \sum_{l\neq0} |\eps m + l|^{-2\alpha}|l|^{-2\nu} \norm{\qbar_{m+l/\eps}}_{H^1}^{2\nu}\;,
\end{align*} 
where the last inequality follows from Lemma \ref{l.split} and the smoothness of $\rho$. If $\alpha \in (1/2,1)$, then set $\nu=0$, if $\alpha \in (0,1/2]$, then set $\nu=1$. In either case, the above sum is bounded uniformly in $\eps$ and $m$, as long as $|m| < \eps^{-1}/2$. For $|m| < \eps^{-\beta}$, we therefore have that  
\begin{align*}
\lim_{\eps\to0} \eps^{-2\alpha} \sum_{l\neq0} \innerprod{q_{m+l/\eps} \rho, e_{-l}}^2 &=  \sum_{l\neq0} \lim_{\eps \to 0} \eps^{-2\alpha} \innerprod{q_{m+l/\eps} \rho, e_{-l}}^2\\ 
&=\sum_{l\neq0} \lim_{\eps \to 0} \eps^{-2\alpha} |m+l/\eps|^{-2\alpha} \innerprod{|m+l/\eps|^{\alpha}q_{m+l/\eps}\rho,e_{-l}}^2 \\
&= \sum_{l\neq0} |l|^{-2\alpha} \innerprod{\qbar \rho, e_{-l}}^2  
= \norm{\qbar \rho}_{-\alpha}^2\;.
\end{align*}
For the first sum in \eqref{e.weaklowmodes}, it is now clear that if $s>0$ then
\[
\sum_{|m|<\eps^{-\beta}}(\Lambda_{\eps,m}-\norm{\qbar \rho}_{-\alpha})^2({1\wedge m^{-(2+2s)}}) \lesssim \sum_{m}({1\wedge m^{-(2+2s)}})\;,
\]
and is therefore bounded uniformly in $\eps$. Hence, we have that 
\begin{equs}
\lim_{\eps \to 0} \sum_{|m|<\eps^{-\beta}}(\Lambda_{\eps,m}&-\norm{\qbar \rho}_{-\alpha})^2({1\wedge m^{-(2+2s)}}) \\
&= \sum_{|m|<\eps^{-\beta}} \lim_{\eps \to 0} (\Lambda_{\eps,m}-\norm{\qbar \rho}_{-\alpha})^2({1\wedge m^{-(2+2s)}})  =0\;.
\end{equs}
For the second sum in \eqref{e.weaklowmodes}, we have that 
\[
\eps^{2-2\alpha-\delta} \sum_{|m|<\eps^{-\beta}} (1\wedge m^{-2s}) \lesssim \eps^{2-2\alpha-\delta} (1\vee \eps^{-(1-2s)\beta}) \;.
\]
For the third sum in \eqref{e.weaklowmodes}, we have that 
\begin{align*}
(\eps^{2\alpha}\vee \eps^{2-2\alpha})& \eps^{-7\delta} \sum_{|m|<\eps^{-\beta}} (1\wedge m^{2-2s+\delta})\\  &\lesssim (\eps^{2\alpha}\vee \eps^{2-2\alpha})\eps^{-7\delta}  (1\vee \eps^{-(3-2s+\delta)\beta}) \\ &\lesssim (\eps^{2\alpha}\vee \eps^{2-2\alpha}) \eps^{-(3-2s+\delta)\beta-7\delta}\;,
\end{align*}
provided $s >0$. For the high modes, we have that
\begin{align*}
\E\sup_{t \in [0,T]}&\sum_{|m| \geq \eps^{-\beta}}  |\innerprod{\eps^{-\alpha}\tildeu_\eps(t) -v(t),e_m}|^2 (1+m^2)^{-s} \\
 &\lesssim 
 \E \suptime \sum_{m \in \integers}   |\innerprod{\tildeu_\eps(t),e_m}|^2 \eps^{2\beta s-2\alpha} +\E \suptime \sum_{m \in \integers} |\innerprod{v(t),e_m}|^2  \eps^{2\beta s} \\ 
 &\lesssim 
 \eps^{2\beta s -2\alpha} \E \suptime \norm{\tildeu_\eps(t)}^2 + \eps^{2 \beta s} \E \suptime \norm{v(t)}^2 \\
& \lesssim \eps^{2\beta s - 2\alpha} (1\vee \eps^{4\alpha-2-\delta}) + \eps^{2\beta s}\;.
\end{align*}
Here we have used Lemma \ref{l.apriori} as well as the clear fact that
\[
\E \sup_{t \in [0,T]} \norm{v(t)}^2  \lesssim 1\;.
\]
If $\alpha \in (1/2,1)$, then for both the low and high modes to converge to zero for some $s >0$, we need to find $\beta \in (0,1)$ such that 
\[
\frac{\alpha}{s} < \beta < \frac{2-2\alpha}{3-2s}\;.
\]
A simple diagram confirms that we can always find such a $\beta$ provided $s > {3\over 2}\alpha$.
If $\alpha \in (0,1/2]$, then for both the low and high modes to converge to zero for some $s>0$, we need to find $\beta \in (0,1)$ such that
\[
\frac{1-\alpha}{s} < \beta < \frac{2\alpha}{3-2s}\;.
\]
A simple diagram confirms that we can always find such a $\beta$, provided $s > {3\over 2}(1-\alpha)$. This 
concludes the proof of the theorem.
\end{proof}

Before proving Theorem \ref{thm.STWN}, we need a new a priori bound on the solution $u_\eps$, given that we are working with new assumptions on the noise. 

\begin{lemma}\label{lem.apriori2}
Suppose $u_\eps$ satisfies \eqref{e:formulationSPDE} and the conditions given in Assumptions \ref{ass:elliptic}, \ref{ass.STWN} hold true, then we have that
\begin{equation}\label{e:apriori2e1}
\E \suptime \norm{u_\eps(t) }^2 \lesssim \eps^{-2-\delta}\;,
\end{equation}
for arbitrarily small $\delta>0$. 
\end{lemma}
\begin{proof}
From Lemma \ref{l.apriori} we know that 
\[
\E \suptime \norm{u_\eps(t)}^2 \lesssim \int_0^T \norm{S_\eps(t)Q_\eps}_\HS^2 \,dt\;.
\]
We can bound the Hilbert-Schmidt norm using Assumption \ref{ass.STWN}. We have that
\[
\norm{S_\eps(t) Q_\eps}^2_\HS = \sum_k \norm{S_\eps(t)q_k^\eps e_k}^2 \lesssim \sum_k \norm{S_\eps(t)(q_k^\eps-\qbar^\eps) e_k}^2 + \sum_k \norm{S_\eps(t)\qbar^\eps e_k}^2\;.
\]
But the first term can be bounded
\[
\sum_k \norm{S_\eps(t)(q_k^\eps-\qbar^\eps) e_k}^2 \lesssim \eps^{-4\nu} |t|^{-\nu} \sum_k |k|^{-2\nu} \norm{q_k-\qbar}_{H^1}^2  \;,
\]
for any $\nu \in [0,1)$ using the same argument found in Lemma \ref{l.apriori}. By assumption, the sum over $k$ is finite if we set $2\nu=\eta$. For the second term, we similarly know that 
\[
\sum_k \norm{S_\eps(t)\qbar^\eps e_k}^2 \lesssim \eps^{-4\gamma} |t|^{-\gamma} \norm{\qbar}_{H^1} \sum_k |k|^{-2\gamma}\;,
\]
for any $\gamma \in (0,1)$. If we set $\gamma = 1/2+\delta/4$, for arbitrarily small $\delta>0$, then the sum over $k$ will converge. Since $2\eta < 2$, the $\eps^{-4\gamma}$ term will be the dominant one. It follows that 
\[
\int_0^T \norm{S_\eps(t)Q_\eps}_\HS dt \lesssim \eps^{-4\gamma} = \eps^{-2 - \delta }\;.
\] 
This proves the lemma.   
\end{proof}

\begin{proof}[Proof of Theorem \ref{thm.STWN}] As in the proof of Theorem~\ref{thm:weaklimit}, we construct 
sequences $\{\hatW_m\}$ and $\{B_k^\eps\}$ of Brownian motions with correlations 
\begin{equation}
\E \hatW_m(t) B^\eps_k(s) = 
\begin{cases}
\frac{\lambda_{\eps,m}^l}{\Lambda_{\eps,m}}(t\wedge s), &\text{if $k = m+l/\eps$ for some $l\in \integers$}\;, \\
0 , &\text{otherwise} \;,
\end{cases}
\end{equation}
where $\lambda_{\eps,m}^l = \innerprod{q_{m+l/\eps}e_l,\rho}$  and, as before, $\Lambda_{\eps,m} = \left(\sum_{l\in\integers} |\lambda_{\eps,m}^l|^2\right)^{1/2}$.  We then define $\hatu_\eps$ through its Fourier modes as follows For $|m| \le \eps^{-\beta}$, we set 
\[
\innerprod{\hatu_\eps(t),e_m} = \Lambda_{\eps,m} \int_{0}^{t} e^{-\mu m^2 (t-s)} d{\hatW}_m(s)\;,
\] 
while for $|m| > \eps^{-\beta}$, we set 
\[
\innerprod{\hatu_\eps(t),e_m} = \innerprod{w_\eps(t),e_m}\;,
\]
where $w_\eps$ solves \eqref{e:formulationSPDE} with $W_k$ replaced with $B_k^\eps$ for each $k \in \integers$. This is identical to the construction given in the proof of Theorem~\ref{thm:weaklimit}, with the sole difference being that now $\lambda_{\eps,m}^0 \neq 0$, in general. The proof proceeds identically to the previous theorem. We only need a few more ingredients to ensure that this proof will work just like the last. First, we need that 
\[
\Lambda_{\eps,m}- (|\innerprod{q_m,\rho}|^2 - |\innerprod{\qbar,\rho}|^2 + \norm{\qbar \rho}^2)^{1/2}
\] 
converges to zero as $\eps \to 0$. But this is true by construction of the series $\Lambda_{\eps,m}$, using the same arguments as previously employed to pass the limit inside the sum. Secondly, we need some bound on the remainder terms of the low modes. We cannot use the previous bounds \eqref{e.claim2} and \eqref{e.claim3}, since we are effectively using $\alpha=0$. However, just as in Lemma \ref{lem.apriori2} we can use Assumption \ref{ass.STWN} instead. We claim the following bounds to be true and prove them in the sequel. For $|m| \le \eps^{-\beta}$, we have that 
\begin{align}
\E \suptime \left| \sum_{k} \int_0^t \innerprod{q_k^\eps e_k,\rhobar^\eps_{(t-s)/\eps^2}e_m}dB^\eps_{k}(s)\right|^2 & \lesssim \eps^{2-\eta - 2\delta} |m|^\eta\;, \label{e.thm3claim1} \\
\E \suptime \left|\sum_k \int_0^t \innerprod{q_k^\eps e_k,R_\eps(t-s)}dB^\eps_k(s)\right|^2 & \lesssim \eps^{2-2\eta-3\delta} |m|^{2+\delta}\;\label{e.thm3claim2},
\end{align}
for arbitrarily small $\delta >0$. From Lemma \ref{lem.apriori2} we have that 
\[
\E \suptime \norm{\hatu_\eps(t)}^2 \lesssim \eps^{-2-\delta}\;.
\]
Moreover, one can easily show that 
\[
\E \suptime \norm{\hatu(t)}^2 \leq C_T\;.
\]
We can then apply the exact arguments used in Theorem \ref{thm:weaklimit} to show that both high and low modes will converge to zero as $\eps \to 0$ if we can choose $\beta \in (0,1)$ in such a way that
\[
\frac{1}{s} < \beta < \frac{2-2\eta}{3-2s}\;.
\]
It is easy to show that one can always find such a $\beta$ provided $s>s_\eta$, where 
\[
s_\eta = 1 \vee \frac{3}{2(2-\eta)}\;.
\]
This proves \eqref{e.STWNlimit}. We now prove the claimed bounds. For \eqref{e.thm3claim1} and \eqref{e.thm3claim2} we apply the Kolmogorov criterion from Lemma \ref{l.kolmogorov} just as we did to bound \eqref{e.claim2} and \eqref{e.claim3} respectively. This involves proving four estimates (two for each claim). For the first claim, we wish to find $K_\eps(\delta)$ such that 
\begin{equation}\label{e.thm3claim1proof1}
\E \left| \sum_{k} \int_s^t \innerprod{q_k^\eps e_k,\rhobar^\eps_{(t-r)/\eps^2}e_m}dB^\eps_{k}(r)\right|^2 \leq K_\eps(\delta) |t-s|^\delta\;,
\end{equation}
and 
\begin{equation}\label{e.thm3claim1proof2}
\E \left| \sum_{k} \int_0^t \innerprod{q_k^\eps e_k,(\rhobar^\eps_{(t-r)/\eps^2} - \rhobar^\eps_{(t-s)/\eps^2})e_m}dB^\eps_{k}(r)\right|^2 \leq K_\eps(\delta) |t-s|^\delta\;,
\end{equation}
for some $\delta \in (0,1)$. Clearly, we can bound the left hand side of \eqref{e.thm3claim1proof1} by a constant multiple of
\[
\sum_k \int_s^t |\innerprod{(q_k^\eps-\qbar^\eps)e_k,\rhobar^\eps_{(t-r)/\eps^2}e_m}|^2 dr + \sum_k \int_s^t |\innerprod{\qbar^\eps e_k,\rhobar^\eps_{(t-r)/\eps^2}e_m}|^2 dr\;.
\]
Applying Lemma \ref{l.split} (with $2\nu=\eta$) to the first term and using the fact that, for every $m$, one has $\sum_k |\innerprod{e_ke_{-m},f}|^2 = \norm{f}^2$ for the second term, we can bound this by 
\begin{equs}
\eps^{-\eta} &\left( \sum_k |m-k|^{-\eta} \norm{q_k-\qbar}_{H^1}^2 \right)  \int_s^t \norm{\rhobar_{(t-r)/\eps^2}}_{H^1}^\eta \norm{\rhobar_{(t-r)/\eps^2}}^{2-\eta} dr \\
&\qquad +  \int_s^t \norm{\qbar^\eps \rhobar^\eps_{(t-r)/\eps^2}}^2 dr\;, \\
&\lesssim \eps^{-\eta}|m|^\eta \left( \sum_k |k|^{-\eta} \norm{q_k-\qbar}_{H^1}^2\right) \int_s^t \norm{\rhobar_{(t-r)/\eps^2}}_{H^1}^\eta \norm{\rhobar_{(t-r)/\eps^2}}^{2-\eta} dr\\
&\qquad +  \norm{\qbar}_\infty \int_s^t \norm{\rhobar_{(t-r)/\eps^2}}^2 dr\;.
\end{equs}
By Assumption \ref{ass.STWN} the sum over $k$ is finite and, by a Sobolev embedding, $\norm{\qbar}_\infty$ is also finite. The integral terms can be bounded exactly as in the proof of estimate \eqref{e.claim2} to obtain $K_\eps(\delta) = \eps^{2-\eta -2\delta}|m|^\eta$. We then treat \eqref{e.thm3claim1proof2}, and also the two respective estimates required to prove \eqref{e.thm3claim2} in the same way, by first splitting $q_k$ into $(q_k - \qbar) + \qbar$ and then applying the results from the proof of \eqref{e.claim2} and \eqref{e.claim3}. The estimates \eqref{e.thm3claim1} and \eqref{e.thm3claim2} follow.  
\end{proof}
\begin{proof}[Proof of Corollary \ref{corr:convolution}]
The proof follows in the same way as that of Theorem \ref{thm.STWN}, except we now have $\lambda_{\eps,m}^l = \vphi(\eps m + l) \innerprod{q_{m+l/\eps} e_l ,\rho}$. Moreover, we now need to show that
\begin{equation}\label{e:corr1}
\Lambda_{\eps,m}- (|\innerprod{q_m,\rho}|^2 - |\innerprod{\qbar,\rho}|^2 + \norm{(\qbar \rho)\star \tilde{\vphi}}^2)^{1/2}
\end{equation}
converges to zero as $\eps \to 0$, where $\Lambda_{\eps,m}$ is defined as above, using the new $\lambda_{\eps,m}^l$. But it is clear that
\begin{align*}
\Lambda_{\eps,m}^2 &=  |\vphi(\eps m)|^2 |\innerprod{q_m,\rho}|^2 + \sum_{ l\neq0}   |\vphi(\eps m+l)|^2 |\innerprod{q_{m+l/\eps}e_l,\rho}|^2\\
&\quad \to |\innerprod{q_m,\rho}|^2 + \sum_{l\neq0} |\vphi(l)|^2 |\innerprod{\qbar \rho,e_l}|^2\;,
\end{align*}
where the boundedness of $\vphi$ in combination with previous arguments allows us to take the limit inside the sum over $l$. Since $\norm{(\qbar \rho) \star \tilde{\vphi}}^2 = \sum_{l \in \integers} |\vphi(l)|^2 |\innerprod{\qbar \rho,e_l}|^2$, we have proven \eqref{e:corr1}. The remainder of the proof follows in exactly the same way as Theorem \ref{thm.STWN}, and since $\vphi$ is bounded, all corresponding estimates still hold. 
\end{proof}

\bibliographystyle{./Martin}
\bibliography{./homogen}

\end{document}